\documentclass{article}

\usepackage[english]{babel} 
\usepackage{enumitem}
\usepackage[utf8]{inputenc} 
\usepackage[a4paper,left=2cm,right=2cm,bottom=3cm]{geometry} 
\usepackage{indentfirst} 
\usepackage{amsmath,amssymb,amsfonts,amsthm} 
\numberwithin{equation}{section}
\usepackage{bm} 
\usepackage{graphicx} 
\usepackage[export]{adjustbox} 
\usepackage{pdflscape} 
\usepackage{fancyhdr} 
\usepackage{natbib} 
\setcitestyle{numbers,square}
\usepackage{flafter} 
\usepackage[framemethod=tikz]{mdframed} 
\usepackage{color} 
	\definecolor{yellow-green}{rgb}{0.6, 0.8, 0.2}
	\definecolor{viridian}{rgb}{0.25, 0.51, 0.43}
\usepackage{wrapfig} 
\usepackage{lipsum} 

\usepackage[all]{xy}
\usepackage[auth-lg,noblocks]{authblk}
\usepackage{framed}
\usepackage{tikz}
\usepackage{quiver}
\usetikzlibrary{hobby}

\usepackage{mathrsfs}

\usepackage{dsfont}

\allowdisplaybreaks[1]

\graphicspath{ {Images/} } 

\usepackage{multicol} 
\usepackage{float} 

\usepackage{nicematrix}
\usetikzlibrary{patterns}
\usetikzlibrary{matrix,decorations.pathreplacing}
\usetikzlibrary{decorations.pathreplacing,calligraphy}
\usepackage[colorlinks=true,linkcolor=cyan,citecolor=viridian]{hyperref} 
\hypersetup{
    pdfcreator={},
    pdfproducer={LaTeX},
    hypertexnames=false,  
    linktocpage=true,
    colorlinks=true
}
\usepackage{prettyref}
\usepackage[nameinlink]{cleveref}
\title{Non-Hitchin Borel Anosov representations from surface groups to $\mathrm{SL}_{3k}\mathbb{R}$}

\author{Zhufeng Yao\thanks{Department of Mathematics, National University of Singapore, 119077, Singapore, \href{mailto:yaozhufeng@u.nus.edu}{yaozhufeng@u.nus.edu}}  \mbox{ and} Junming Zhang\thanks{Chern Institute of Mathematics and LPMC, Nankai University, Tianjin 300071, China, \href{mailto:junmingzhang@mail.nankai.edu.cn}{junmingzhang@mail.nankai.edu.cn}}}

\date{\vspace{-2em}}

\newtheorem{theorem}{Theorem}[section]
\newtheorem{corollary}[theorem]{Corollary}
\newtheorem{proposition}[theorem]{Proposition}
\newtheorem{lemma}[theorem]{Lemma}
\newtheorem{definition}[theorem]{Definition}
\newtheorem{remark}[theorem]{Remark}

\newrefformat{lem}{Lemma \ref{#1}}
\newrefformat{coro}{Corollary \ref{#1}}
\newrefformat{thm}{Theorem \ref{#1}}

\newrefformat{prop}{Proposition \ref{#1}}
\newrefformat{rem}{Remark \ref{#1}}
\newrefformat{example}{Example \ref{#1}}
\newrefformat{defn}{Definition \ref{#1}}
\newrefformat{section}{Section \ref{#1}}
\newrefformat{conj}{Conjecture \ref{#1}}
\newrefformat{apdx}{Appendix \ref{#1}}
\newrefformat{fact}{Fact \ref{#1}}
\newcommand{\dd}{{\mathrm d}}
\newcommand{\EE}{{\mathcal{E}}}
\newcommand{\FF}{{\mathcal{F}}}
\newcommand{\KK}{{\mathcal{K}}}

\newcommand{\sslash}{
	\mathchoice{\mathbin{\mkern-3mu/\mkern-6mu/\mkern-3mu}}
	{\mathbin{\mkern-3mu/\mkern-6mu/\mkern-3mu}}
	{\mathbin{\mkern-2mu/\mkern-5mu/\mkern-1mu}}
	{\mathbin{\mkern-2mu/\mkern-5mu/\mkern-1mu}}
}
\newcommand{\iu}{{\mathrm i}}

\DeclareMathOperator{\End}{End}

\begin{document}

\pagenumbering{gobble} 
\maketitle

\pagenumbering{arabic} 
\setcounter{section}{0}
\setcounter{page}{1}
\vspace{-1em}

\begin{abstract}
We use the Labourie--Wentworth's formula and the thermodynamic formalism to show that, along the slice constructed by Bronstein--Davalo~\cite{bronstein2025anosov}, the logarithmic top-eigenvalue length spectrum has a uniformly positive second variation near the Barbot representation. 

As a major application, we show that every closed surface group admits a non-Hitchin Borel Anosov representation into $\mathrm{SL}_{3k}\mathbb{R}$ for every $k\geqslant 1$. In particular, we obtain the first such examples in the even dimensions $6k$. We also study the local behavior of related objects of this slice around the Barbot representation, including the Lyapunov exponent of the flat bundle, the Hausdorff dimension of the limit set, and the Hilbert entropy of the representation.
\end{abstract}



\large

\tableofcontents

\section{Introduction}
Let $S$ be a closed oriented surface of genus at least $2$. Using the theory of Higgs bundles, Hitchin~\cite{HITCHIN1992449} identified some distinguished connected components of the character variety \[\mathfrak{X}(S,\mathrm{SL}_n\mathbb{R}):=\mathrm{Hom}(\pi_1(S), \mathrm{SL}_n\mathbb{R})\sslash\mathrm{SL}_n\mathbb{R},\] now known as the \emph{Hitchin component}. For $n=2$, this canonically recovers the classical Teichmüller space of $S$. More generally, the Hitchin component is homeomorphic to a real cell, offering a natural framework to extend Teichmüller theory to higher-rank Lie groups and laying the foundation for higher Teichmüller theory. The elements within this component are referred to as Hitchin representations. 

Labourie \cite{labourie2006anosov} and Fock--Goncharov \cite{fock2006modulispaceslocalsystems} proved that Hitchin representations are discrete and faithful independently. Moreover, Labourie showed that Hitchin representations are irreducible and possess a strong dynamical property known as the \emph{Anosov property}. These results position Hitchin representations as natural higher-rank analogs of Fuchsian representations, mirroring many of the geometric and dynamical features of classical Teichmüller theory.

We now introduce the concept of Anosov representations more concretely. Let $\lambda_i(g)$ denote the eigenvalues of $g\in \mathrm{SL}_n\mathbb{R}$, listed in non-increasing order of their moduli. For $g\in \mathrm{SL}_2 \mathbb{R}$, we use $\lambda(g)$ to abbreviate $\lambda_1(g)$. 

Equip the surface $S$ with a Riemann surface structure $X$. Let $g_{X}$ be the conformal hyperbolic metric on $X$ with curvature $-1$. We denote by $\KK_X$ the canonical line bundle of $X$. We fix a square root $\KK_X^{1/2}$ of $\KK_X$, which is equivalent to a spin structure of $X$. This gives us a lift $j_X\colon\pi_1(S)\to\mathrm{SL}_2\mathbb{R}$ of the holonomy representation of $g_X$. For any $\gamma\in \pi_1(S)$, we define $\ell_X(\gamma) := 2\log |\lambda(j_X(\gamma))|$. For convenience of this paper, we define Anosov representations in reference to the work by Kassel and Potrie: 

\begin{definition}[{\cite[Proposition~1.5]{kasselpotrieeigenvaluegap}}]\label{defn:eigenvalue-gap}
Let $1\leqslant k\leqslant n-1$. A representation $\rho:\pi_1(S)\to \mathrm{SL}_n\mathbb{R}$ is called \emph{$k$-Anosov} if there exist constants $a,A>0$ such that for any $\gamma\in \pi_1(S)$,
\[
\log \frac{|\lambda_k(\rho(\gamma))|}{|\lambda_{k+1}(\rho(\gamma))|} > a \ell_X(\gamma) - A.
\]
\end{definition}

Note that for any two complex structures $X_1$ and $X_2$ on $S$, there exist constants $a'>1$ and $A'>0$ such that for all $\gamma\in \pi_1(S)$,
\[
\frac{1}{a'} \ell_{X_1}(\gamma)-A' \leqslant \ell_{X_2}(\gamma)\leqslant  a'\ell_{X_1}(\gamma)+A'.
\]
As a consequence, the choice of the underlying complex structure $X$ does not affect the definition of a $k$-Anosov representation. Note also that $k$-Anosovness directly implies $(n-k)$-Anosoveness, so there is only one convention for Anosov representations to $\mathrm{SL}_3 \mathbb{R}$. For further foundational and modern developments on Anosov representations, we refer the reader to Guichard--Wienhard~\cite{GWAnosov}, Kapovich--Leeb--Porti~\cite{kapovich2014morse, kapovich2017anosov, anosovmorseandbuilding}, Gu\'eritaud--Guichard--Kassel--Wienhard~\cite{Gu_ritaud_2017}, Bochi--Potrie--Sambarino~\cite{dominatedsplitting}, and many others.

A representation $\rho:\pi_1(S)\to \mathrm{SL}_n\mathbb{R}$ is called \emph{Borel Anosov} if it is $k$-Anosov for every $1 \leqslant k \leqslant n-1$. As previously mentioned, Labourie~\cite{labourie2006anosov} established that Hitchin representations into $\mathrm{SL}_n\mathbb{R}$ are Borel Anosov. Determining whether there exist Borel Anosov representations in $\mathrm{SL}_n\mathbb{R}$ that do not belong to the Hitchin component remains an interesting problem (see Canary~\cite[Question~50.6]{canary2021anosov} and related discussions in Canary--Tsouvalas~\cite{canary2020topological}). When $n > 1$ is odd, such examples can be easily constructed by taking the direct sum of the symmetric power $\operatorname{Sym}^{n-2}(j_X)$ of the Fuchsian representation $j_X$ with the trivial representation $\mathbf{1}:\pi_1(S)\to \mathbb{R}^*$. Conversely, when $n > 2$ is even, the question remains open. Below is a partial list of known results concerning the restrictions on even-dimensional Borel Anosov subgroups. 

\begin{enumerate}
\item \cite{canary2020topological} proved that every Borel Anosov representation of a closed surface group into
$\mathrm{SL}_4\mathbb{R}$ is irreducible. They also proved that every
torsion-free hyperbolic group admitting a Borel Anosov representation into $\mathrm{SL}_3\mathbb{R}$ or $\mathrm{SL}_4\mathbb{R}$ is either a free group or a closed surface group. 

\item Tsouvalas~\cite{Tsouvalasevenborelanosov} proved that every hyperbolic group admitting a Borel Anosov representation into
$\mathrm{SL}_{4k+2}\mathbb{R}$ is either virtually free or virtually a closed surface group. 

\item Davalo~\cite{davalomaximal} proved that every maximal Borel Anosov representation of a closed surface group into $\mathrm{Sp}_4\mathbb{R}$ is a Hitchin representation.
\end{enumerate}

In this paper, we will construct a family of Borel Anosov representations from a surface group into $\mathrm{SL}_{6k}\mathbb{R}$ that are not Hitchin for any $k\in\mathbb{N}^*$. Our starting point is the \emph{Barbot representation} $(j_X\oplus\mathbf{1})\colon\pi_1(S)\to\mathrm{SL}_3\mathbb{R}$, which is Borel Anosov. In \cite{bronstein2025anosov}, Bronstein--Davalo studied a slice deforming from it, consisting entirely of Borel Anosov representations. 

The definition of the Bronstein--Davalo slice involves the Higgs bundle and the non-Abelian Hodge correspondence. For a closed oriented hyperbolic surface $S$ equipped with a Riemann surface structure $X$, by the celebrated non-Abelian Hodge correspondence proved by Hitchin \cite{hitchin1987self}, Donaldson \cite{donaldson1987twisted}, Corlette \cite{corlette1988flat}, Simpson \cite{simpson1988constructing,simpson1992higgs}, reductive representations $\pi_1(S)\to\mathrm{GL}_n\mathbb{C}$ correspond to polystable $\mathrm{GL}_n\mathbb{C}$-Higgs bundles, which is a holomorphic concept consisting of a holomorphic vector bundle $\EE$ with rank $n$, degree $0$ and a Higgs field $\Phi\in\mathrm{H}^0(X,\operatorname{End}(\EE)\otimes\KK_X)$. Moreover, the non-Abelian Hodge correspondence $\operatorname{NAH}$ establishes a homeomorphism between the moduli space of polystable Higgs bundles and the character variety, which is real-analytic in the smooth part.

Bronstein--Davalo studied the following cyclic Higgs bundle 
 \[(\EE,\Phi_q):=\left(\KK_X^{1/2}\oplus\mathcal{O}_X\oplus\KK_X^{-1/2},\begin{pmatrix}
     &q&\\&&q\\\mathds{1}&&
 \end{pmatrix}\right),\] 
 where $\mathds{1}\colon\KK_X^{1/2}\to\KK_X^{-1/2}\otimes\KK_X$ is the tautological isomorphism and $q\in\mathrm{H}^0(X,\KK_X^{3/2})$. When $q=0$, $\operatorname{NAH}(\EE,\Phi_0)$ is the Barbot representation $j_X\oplus\mathbf{1}$. For every $q$, they proved the following theorem:
\begin{theorem}[{\cite[Theorem 1.1]{bronstein2025anosov}}]\label{thm:BD}
    Through the non-Abelian Hodge correspondence, \[\operatorname{NAH}(\EE,\Phi_q)\colon\pi_1(S)\to\mathrm{SL}_3\mathbb{R}\] is Borel Anosov for any $q\in\mathrm{H}^0(X,\KK_X^{3/2})$.
\end{theorem}

Below, we fix a nonzero holomorphic differential $q$ and study the behavior of the top eigenvalue spectrum of $\rho_{t}:=\operatorname{NAH}(\EE,\Phi_{tq})$, where $t\in\mathbb{R}$, when $|t|$ is sufficiently small. More precisely, we prove the following domination theorem, which plays an important role in the applications below:

\begin{theorem}\label{thm:Hilbert domination}
There exists $\varepsilon_0=\varepsilon_0(q)>0$ and $m_0 = m_0(q)>  0$  such that for any $t\in(-\varepsilon_0,\varepsilon_0)$ and $\gamma\in \pi_1(S)$,
\[
\log |\lambda_1(\rho_t(\gamma))|\geqslant \left(\frac{1}{2}+m_0t^2\right)\cdot\ell_X(\gamma), \quad -\log |\lambda_3(\rho_t(\gamma))|\geqslant \left(\frac{1}{2}+m_0t^2\right)\cdot\ell_X(\gamma).
\]
\end{theorem}

We note that as a corollary, this proves that the Hilbert length of $\rho_t$ dominates half of the hyperbolic length, i.e. for any $|t|\in (0,\varepsilon_0)$ there exists a constant $K_t=K(t,q)>1$ such that\[\dfrac{1}{2}\log \dfrac{|\lambda_1(\rho_t(\gamma))|}{|\lambda_3(\rho_t(\gamma))|}\geqslant K_t\cdot\dfrac{\ell_X(\gamma)}{2}.\] For Hitchin representations to $\mathrm{SL}_3 \mathbb{R}$, the length domination between the Hilbert length and hyperbolic length is proven by Tholozan~\cite{tholozandomination}.

We get a sequence of applications of \prettyref{thm:Hilbert domination}. 

\paragraph{Non-Hitchin Borel Anosov representations}

Firstly, we have the following corollary which constructs non-Hitchin Borel Anosov representations from $\pi_1(S)$ to $\mathrm{SL}_{3k}\mathbb{R}$ for any closed oriented hyperbolic surface $S$ and any $k\in\mathbb{N}^*$. In particular, this gives the first family of examples of non-Hitchin Borel Anosov representations from $\pi_1(S)$ to $\mathrm{SL}_{6k}\mathbb{R}$.

\begin{corollary}\label{coro: 6dimnonhitchin}
    Let $k>1$ and $\operatorname{Sym}^{k-1}(j_X)\colon\pi_1(S)\to\mathrm{SL}_{k}\mathbb{R}$ be the $(k-1)$-th symmetric power of the Fuchsian representation $j_X$, then there exists $\varepsilon_1=\varepsilon_1(q)>0$ such that for any $t\in \mathbb{R}, |t|\in(0,\varepsilon_1)$, \[\left(\rho_t\otimes\operatorname{Sym}^{k-1}(j_X)\right)\colon\pi_1(S)\to\mathrm{SL}_{3k}\mathbb{R}\] is irreducible, Borel Anosov but non-Hitchin.
\end{corollary}

\begin{remark}\label{rem:non-borel-anosov}
Note that whenever $k>1$, the representation $\rho_0\otimes\operatorname{Sym}^{k-1}(j_X)$ is neither irreducible nor Borel Anosov. Indeed, one can directly check that for any nontrivial $\gamma\in \pi_1(S)$, $\lambda^{k-2}(j_X(\gamma))$ is always an eigenvalue of multiplicity at least $2$. As a result, the representation cannot be Borel Anosov. This gives an example of robust quasi-isometric embedding which is a limit of Borel Anosov representations, but non-Borel Anosov itself. One could compare this with \cite[Question 5]{potrie2018robust}, \cite[Section 8]{kassel2018geometric} and \cite[Corollary 5.5]{tsouvalas2026robust}.
\end{remark}

\paragraph{Lyapunov exponents}

Next, let $\pi:\mathbb H^2\to S$ be the universal covering associated with the Fuchsian representation $j_X:\pi_1(S)\to\mathrm{SL}_2\mathbb R$, and let $\pi^T:T^1\mathbb H^2\to T^1S$ be the induced covering. For a representation $\rho:\pi_1(S)\to\mathrm{SL}_3\mathbb R$, define the flat bundle
\[
\mathsf F_{X,\rho}:=\pi_1(S)\backslash (T^1\mathbb H^2\times\mathbb R^3),
\quad \text{where group action is given by }
\gamma\cdot(v,u):=(j_X(\gamma)v,\rho(\gamma)u).
\]
The geodesic flow $\varphi^{\tau}$ induces the linear flow
\[
\varphi_{X,\rho}^\tau[v,u]=[\varphi^\tau v,u]
\]
on $\mathsf F_{X,\rho}$. With respect to the normalized Liouville measure $m_X$ which is ergodic, denote its Lyapunov exponents by
\[
L_1({X,\rho})\geqslant L_2({X,\rho})\geqslant L_3({X,\rho}).
\] 
Using \prettyref{thm:Hilbert domination} we show
\begin{corollary}\label{coro:lyapunov exponent}
There exist $\varepsilon_2 = \varepsilon_2(q)>0$ and $m_2 = m_2(q)>0$ such that for any $t\in (-\varepsilon_2,\varepsilon_2)$,
\[
L_1(X,\rho_t)\geqslant \frac{1}{2}+m_2 t^2,\quad L_3(X,\rho_t)\leqslant -\frac{1}{2}-m_2 t^2.
\]
\end{corollary}

\paragraph{Hausdorff dimension of the limit set}

Recall that if $\rho: \pi_1(S)\to \mathrm{SL}_3\mathbb{R}$ is an Anosov representation, then we can define the projective limit set $\Lambda^1_\rho\subset \mathbb{R}\mathbb{P}^2$ by taking the closure of the attractive eigenlines of $\rho(\gamma),\gamma\in \pi_1(S)\setminus \{\mathrm{id}\}$. Combined with a result of Li--Pan--Xu~\cite[Theorem~1.2]{lipanxu} (see also Pozzetti--Sambarino--Wienhard~\cite[Theorem~B]{liplim} ), \prettyref{thm:Hilbert domination} can be applied to obtain the following:

\begin{corollary}\label{coro: affinity exponent}
There exist $\varepsilon_3 = \varepsilon_3(q)>0$, $m_3 = m_3(q)>0$ and $m'_3 = m'_3(q)>0$ such that for any $t\in \mathbb{R}$ with $|t|\in (0,\varepsilon_3)$,
\[
\frac{3}{2} - m'_3t^2\leqslant \dim_{\mathrm{H}} \Lambda^1_{\rho_t} \leqslant \frac{3}{2} - m_3t^2,
\]
where $\dim_{\mathrm{H}}$ denotes the Hausdorff dimension.
\end{corollary}

On the other hand, inspired by \cite{portideform}, we compute an example of arbitrarily small irreducible deformations $\rho:\pi_1(S)\to \mathrm{SL}_{3}\mathbb{R}$ of $\rho_0 = j_X\oplus \mathbf{1}$ in \prettyref{rem: dimension>3/2} satisfying that
\[
\dim_{\mathrm{H}} (\Lambda^1_\rho)>\frac{3}{2}.
\]
Since the irreducible locus in a sufficiently small neighbourhood of $[\rho_0]$ in $\mathfrak{X}(S,\mathrm{SL}_3\mathbb{R})$ is path-connected (see \cite[Theorem~1.1]{portideform}) and the Hausdorff dimension varies continuously there (see \cite[Theorem~1.2]{lipanxu}), we have
\begin{theorem}
For every open neighbourhood $U_{\rho_0}\subset \operatorname{Hom}_{\mathrm{Anosov}}(\pi_1(S), \mathrm{SL}_3 \mathbb{R})$ of $\rho_0$ in the space of Anosov representations, the set $\left\{ \dim_{\mathrm{H}} \Lambda^1_{\rho} \ \middle| \ \rho\in U_{\rho_0}  \right\}$ contains an open neighbourhood of $3/2$. 
\end{theorem}
Li--Pan--Xu~\cite[Theorem~1.3~(2)]{lipanxu} showed that the Hausdorff dimension takes values near $3/2$ in the irreducible locus near $\rho_0$. The above theorem strengthens this observation by showing that $3/2$ is an interior point of the local range.

\paragraph{Hilbert Entropy}
Finally, for any Anosov representation $\rho:\pi_1(S)\to \mathrm{SL}_3\mathbb{R}$, we define the \emph{Hilbert entropy} of $\rho$ by
\[
\delta^{\mathrm{H}}_{\rho} := \limsup_{T\to \infty} \frac{1}{T} \log \# \left\{ [\gamma]\ \middle|\  \frac{1}{2}\log \frac{|\lambda_1(\rho(\gamma))|}{|\lambda_3(\rho(\gamma))|} \leqslant T \right\},
\]
where $[\gamma]$ denotes the conjugacy classes in $\pi_1(S)$. When $\rho$ is Hitchin, it preserves a properly convex domain $\Omega$ in $\mathbb{R}\mathbb{P}^2$, and the Hilbert entropy coincides with the topological entropy of the geodesic flow on the unit tangent bundle of $\rho(\pi_1(S))\backslash \Omega$ with respect to the Hilbert metric on $\Omega$ (see Choi--Goldman~\cite{choigoldman} and Crampon~\cite{cramponentropy}).

Since $\rho_0 = j_X\oplus \mathbf{1}$, $\delta^{\mathrm{H}}_{\rho_0}$ coincides with twice the topological entropy of the geodesic flow on $T^1S$, which is exactly $2$. As a result, \prettyref{thm:Hilbert domination} directly implies the following:

\begin{corollary}\label{coro: entropydrop}
There exist constants $\varepsilon_4 = \varepsilon_4(q)>0$ and $m_4 = m_4(q)>0$ such that for any $t\in (-\varepsilon_4,\varepsilon_4)$,
\[
\delta^{\mathrm{H}}_{\rho_t} \leqslant 2-m_4 t^2.
\]
\end{corollary}

\subsection*{Declaration of AI use}
 The observation in \prettyref{thm:Hilbert domination} and the idea of applying it to construct non-Hitchin Borel Anosov representations from a surface group to $\mathrm{SL}_{2k}\mathbb{R}$ arose first from a conversation between the first author and ChatGPT 5.6 Sol. Following this idea, the proof of \prettyref{thm:Hilbert domination} using Labourie--Wentworth's formula and thermodynamic formalism, together with other applications in this article, are given by the authors. ChatGPT 5.6 Sol also provided the idea of the construction in \prettyref{rem: dimension>3/2}. 

After a retrospective literature review aimed to identify mathematical precedents for the observations suggested by the large language model, the authors found that second-variation mechanisms similar to the one underlying \prettyref{thm:Hilbert domination} appear in the work of Wolf~\cite{wolfwpmetric} and Dai~\cite{daixianvariation}, and the related construction in \prettyref{rem: dimension>3/2} appears in Porti~\cite{portideform}.
 
The text of the manuscript was composed by the authors. All mathematical statements and cited references in the manuscript were independently checked by the authors. The authors take full responsibility for the content of the manuscript.

\subsection*{Acknowledgement}
Z. Yao is supported by the National University of Singapore and the Ministry of Education grant A-8004148-00-00. J. Zhang is supported by NSF of China grant No. 125B2007.
Both authors are also partially supported by the National Key R\&D
Program of China No. 2022YFA1006600, the Fundamental Research Funds for the
Central Universities, and Nankai Zhide foundation. 

Both authors are grateful to Qiongling Li and Tengren Zhang for comments on an earlier version of this paper and their continuous support. Z. Yao also sincerely thanks Qiongling Li for hosting his long-term visit to the Chern Institute of Mathematics from July to December 2026. 

\section{Bronstein--Davalo slice in the Barbot component}

\subsection{Harmonic metric and non-Abelian Hodge correspondence}
Given a closed oriented surface $S$ of genus at least $2$ and a real reductive Lie group $G$, its \emph{character variety} $\mathfrak{X}(S, G)$ to $G$ is the geometric quotient of the space of representations:
 $$\operatorname{Hom}(\pi_1(S),G)\sslash G.$$
 In this expression, $G$ is acting by conjugation, and taking the geometric quotient means that we only consider semi-simple representations, or representations whose orbit by the action by conjugation is closed. In particular taking only this subset of the quotient makes it Hausdorff for the induced topology. 
 
The character variety can be parametrized by a moduli space of objects of analytic nature. This parametrization depends on the choice of a Riemann surface structure $X$ on $S$. The objects in this moduli space are \emph{polystable Higgs bundles}. Here we only consider the case for $G=\mathrm{SL}_n\mathbb{R}$ or $\mathrm{SL}_n\mathbb{C}$.

\begin{definition}
    An \emph{$\mathrm{SL}_n\mathbb{R}$-Higgs bundle} over $X$ is a triple $(\mathcal{E},\Phi, \mathcal{B})$, where
    \begin{itemize}
        \item $\mathcal{E}$ is a holomorphic rank $n$ vector bundle over $X$ with $\det(\mathcal{E})\cong \mathcal{O}_X$,
        \item $\Phi$, which is called \emph{Higgs field}, is a holomorphic section of $ \operatorname{End}(\mathcal{E})\otimes\KK_X$ where $\mathcal{K}_X$ is the canonical line bundle of $X$ with $\operatorname{tr}(\Phi)=0$,
        \item $\mathcal{B}$, which is an \emph{orthogonal structure}, is a holomorphic non-degenerate symmetric pairing of $\mathcal{E}$ so that $\Phi$ is $\mathcal{B}$-symmetric.
\end{itemize}
Without the data of the orthogonal structure $\mathcal{B}$, we have an $\mathrm{SL}_n\mathbb{C}$-Higgs bundle.
\end{definition} 

We call an $\mathrm{SL}_n\mathbb{C}$-Higgs bundle $(\EE,\Phi)$ is \emph{stable} if every $\Phi$-invariant subbundle $\mathcal{F}\neq0,\EE$ of $\mathcal{E}$ satisfies that $\deg(\FF)<0$. There are differences for the real group case. The following definition of stability is given in \cite[Theorem 4.17]{garcia2009hitchin}.

\begin{definition}
An $\mathrm{SL}_n\mathbb{R}$-Higgs bundle $(\mathcal{E}, \Phi,\mathcal{B})$ is 
\begin{itemize}
    \item \emph{stable} if every $\mathcal{B}$-isotropic $\Phi$-invariant subbundle $\mathcal{F}\neq0$ of $\mathcal{E}$ satisfies that $\deg(\FF)<0$;
    \item \emph{polystable} if it is semistable and, for any $\mathcal{B}$-isotropic (resp. $\mathcal{B}$-coisotropic) and $\Phi$-invariant  subbundle $\mathcal{F}_1\neq0,\EE$ such that $\deg(\FF_1) = 0$, there is another $\mathcal{B}$-coisotropic (resp. $\mathcal{B}$-isotropic) and $\Phi$-invariant subbundle $\mathcal{F}_2\neq0,\EE$ such that $\EE\cong\FF_1\oplus\FF_2$.
\end{itemize}
\end{definition}

The non-Abelian Hodge correspondence \cite{donaldson1987twisted,hitchin1987self,corlette1988flat,simpson1988constructing,simpson1992higgs} relate the moduli space of polystable Higgs bundles with the character variety. Here we only explain $G=\mathrm{SL}_n\mathbb{R}$ case. One may see \cite{garcia2009hitchin} for general real reductive Lie groups. Let $\mathcal{M}(X,\mathrm{SL}_n\mathbb{R})$ be the moduli space of polystable $\mathrm{SL}_n\mathbb{R}$-Higgs bundles over $X$ up to orthogonal gauge transformation.

\begin{theorem}[Non-Abelian Hodge correspondence for $\mathrm{SL}_n\mathbb R$]\label{thm:NAHstrong}
Let $(\mathcal E,\Phi,\mathcal B)$ be a polystable $\mathrm{SL}_n\mathbb R$-Higgs bundle. There exists a Hermitian metric $H$ on $\mathcal E$, compatible with $\mathcal B$ and inducing the standard metric on $\det\mathcal E\cong\mathcal O_X$, whose Chern connection $\nabla^H$ satisfies Hitchin's self-duality equation:
\begin{equation}\label{eq:Hitchin}
F(\nabla^H)+[\Phi\wedge\Phi^{*_H}]=0.
\end{equation}
Such a metric is called a \textbf{harmonic metric}. Equivalently, the connection $\mathrm D^H=\nabla^H+\Phi+\Phi^{*_H}$ is flat. Its holonomy is conjugate into $\mathrm{SL}_n\mathbb R$ and determines an element of the character variety $\mathfrak X(S,\mathrm{SL}_n\mathbb R)$. This construction induces a homeomorphism, called the non-Abelian Hodge correspondence:
\[
\operatorname{NAH}\colon\mathcal M(X,\mathrm{SL}_n\mathbb R)\longrightarrow\mathfrak X(S,\mathrm{SL}_n\mathbb R).
\]

Furthermore, if $(\mathcal E,\Phi,\mathcal B)$ is stable, then the harmonic metric is unique. 

\end{theorem}


\subsection{Harmonic metrics for Bronstein--Davalo slice}

Bronstein--Davalo considered a slice in $\mathcal{M}(X,\mathrm{SL}_3\mathbb{R})$ consisting of the Higgs bundles of the following form:
\[(\EE,\Phi_q,\mathcal{B}):=\left(\KK_X^{1/2}\oplus\mathcal{O}_X\oplus\KK_X^{-1/2},\begin{pmatrix}
     &q&\\&&q\\\mathds{1}&&
 \end{pmatrix},\begin{pmatrix}
     0&0&1\\0&1&0\\1&0&0
 \end{pmatrix}\right)\]
in \cite{bronstein2025anosov}, where $\mathds{1}\colon\KK_X^{1/2}\to\KK_X^{-1/2}\otimes\KK_X$ is the tautological isomorphism and $q\in\mathrm{H}^0(X,\KK_X^{3/2})$. We have the stability of these Higgs bundles as $\mathrm{SL}_3\mathbb{R}$-Higgs bundles, which is essentially proved in \cite[Proposition 2.3]{bronstein2025anosov} for them as $\mathrm{SL}_3\mathbb{C}$-Higgs bundles.

\begin{proposition}\label{prop:stability}
    $(\EE,\Phi_q,\mathcal{B})$ is stable as an $\mathrm{SL}_3\mathbb{R}$-Higgs bundle.
\end{proposition}

\begin{proof}
    When $q\neq0$, \cite[Proposition 2.3]{bronstein2025anosov} proved that $(\EE,\Phi_q)$ is stable as an $\mathrm{SL}_3\mathbb{C}$-Higgs bundle, meaning every $\Phi_q$-invariant subbundle of $\EE$ has negative degree. This yields that $(\EE,\Phi_q,\mathcal{B})$ is stable as an $\mathrm{SL}_3\mathbb{R}$-Higgs bundle. When $q\equiv0$, although $(\EE,\Phi_0,\mathcal{B})$ is strictly polystable as an $\mathrm{SL}_3\mathbb{C}$-Higgs bundle, it is still stable as an $\mathrm{SL}_3\mathbb{R}$-Higgs bundle because the only $\Phi_0$-invariant $\mathcal{B}$-isotropic subbundle is $\KK_X^{-1/2}$, which has negative degree.
\end{proof}

Therefore, there exists a unique harmonic metric $H_q$ of $(\EE,\Phi_q,\mathcal{B})$. Since $(\EE,\Phi_q)$ is a $3$-cyclic Higgs bundle (see \cite{simpson2009katz}), $\KK_X^{1/2}\oplus\mathcal{O}_X\oplus\KK_X^{-1/2}$ is an orthogonal decomposition with respect to the harmonic metric. Furthermore, it is shown in \cite[Section 4.1]{bronstein2025anosov} (also that $H_q=\operatorname{diag}(h_q,1,h_q^{-1})$ with respect to the decomposition $\KK_X^{1/2}\oplus\mathcal{O}_X\oplus\KK_X^{-1/2}$, where $h_q$ is an Hermitian metric on $\KK_X^{1/2}$ and $h_q^{-1}$ is its dual metric.

Below we fix a nonzero holomorphic differential $q\in\mathrm{H}^0(X,\KK_X^{3/2})$ and simply denote, for $t\in\mathbb{R}$, by
\[(\EE,\Phi_{t}):=\left(\KK_X^{1/2}\oplus\mathcal{O}_X\oplus\KK_X^{-1/2},\begin{pmatrix}
     &tq&\\&&tq\\\mathds{1}&&
 \end{pmatrix}\right).\]
The corresponding harmonic metric $H_t:=H_{tq}$ satisfies that $H_t=\operatorname{diag}(h_t,1,h_t^{-1})$. Below, we simply write $\nabla_t:=\nabla^{H_t}$, $\Phi_t:=\Phi_{tq}$ and
\[\mathrm{D}_t=\nabla_t+\Phi_t+\Phi_t^{*_{H_t}}\]
as the corresponding flat connection. Let \[\rho_t:=\operatorname{NAH}(\EE,\Phi_t,\mathcal{B})=\operatorname{Hol}(\mathrm{D}_t)\] be the corresponding representation through the non-Abelian Hodge correspondence. As a corollary of \cite[Proposition 2.3]{bronstein2025anosov}, we obtain that 
\begin{corollary}\label{coro:irr}
    $\rho_t$ is irreducible when $t\neq0$.
\end{corollary}

\begin{proof}
    Following from \cite[Proposition 2.3]{bronstein2025anosov}, $(\EE,\Phi_t)$ is stable as an $\mathrm{SL}_3\mathbb{C}$-Higgs bundle when $t\neq0$. Thus it gives an irreducible representation to $\mathrm{SL}_3\mathbb{C}$ through the non-Abelian Hodge correspondence, see for example \cite{simpson1992higgs}. Therefore, $\rho_t$ is irreducible since any real invariant subspace complexified to a complex invariant space.
\end{proof}

\begin{remark}
    It is a little bit confusing that the stable $\mathrm{SL}_3\mathbb{R}$-Higgs bundle $(\EE,\Phi_0,\mathcal{B})$ does not give an irreducible representation to $\mathrm{SL}_3\mathbb{R}$. This is because $(\EE,\Phi_0,\mathcal{B})$ admits non-trivial automorphism $\operatorname{diag}(-1,1,-1)$, which does not lie in the center of $\mathrm{SO}_3\mathbb{C}$. In general, to obtain the irreducibility from Higgs bundle, we also need the simpleness of the Higgs bundle. If we only have the stability, we can only say the corresponding representation if infinitesimally irreducible. See \cite[Theorem 3.32]{garcia2009hitchin}.
\end{remark}

Note that $g_t:=h_t^{-2}$ defines a conformal metric on $X$. Set $u_t$ as the function on $X$ such that $g_t=\mathrm{e}^{2u_t}g_0$. Then $u_0=0$.

Below we do some local computations. Fix a local coordinate $(U,z=x+\iu y)$ on $X$. Then we have the local frame $(\dd z^{1/2},1,\dd z^{-1/2})$ adapted to $\KK_X^{1/2},\mathcal{O}_X,\KK_X^{-1/2}$, the metric $h_t$ can be written as $h_{t}(z)|\dd z|$ for positive smooth functions $h_t(z)$ with a slight abuse of notation. We use $|\bullet|^2$ to denote the coordinate norm for a field with respect to the local coordinate and frames $(\dd z^{1/2},1,\dd z^{-1/2})$. With our choice of frame, we have $|\mathds{1}|^2=1$. 

Now the Hitchin's self-duality equation \prettyref{eq:Hitchin} implies that:
\begin{equation}\label{eq:localHit}
    \dfrac{\partial^2\log h_t}{\partial z\partial\bar z}=|tq|^2h_t-h_t^{-2}.
\end{equation}
by taking projection onto $\KK_X^{1/2}$. Hence taking $t=0$ we obtain that 
\[\dfrac{\partial^2\log h_0}{\partial z\partial\bar z}=-h_0^{-2}.\]
Then by
\[\dfrac{\partial^2\log h_t}{\partial z\partial\bar z}=\dfrac{\partial^2\log(\mathrm{e}^{-u_t}h_0)}{\partial z\partial\bar z}=-\dfrac{\partial^2 u_t}{\partial z\partial\bar z}+\dfrac{\partial^2\log h_0}{\partial z\partial\bar z}\]
and
\[|tq|^2h_t-h_t^{-2}=t^2|q|^2\mathrm{e}^{-u_t}h_0-\mathrm{e}^{2u_t}h_0^{-2},\]
we obtain that
\[\dfrac{1}{h_0^{-2}}\cdot{\dfrac{\partial^2 u_t}{\partial z\partial\bar z}}=\dfrac{1}{h_0^{-2}}\cdot\dfrac{\partial^2\log h_0}{\partial z\partial\bar z}-t^2|q|^2\mathrm{e}^{-u_t}h_0^3+\mathrm{e}^{2u_t}=-1-t^2|q|^2\mathrm{e}^{-u_t}h_0^3+\mathrm{e}^{2u_t}.\]
Note that \[\dfrac{1}{h_0^{-2}}\cdot{\dfrac{\partial^2}{\partial z\partial\bar z}}=\dfrac{1}{g_0}\cdot{\dfrac{\partial^2}{\partial z\partial\bar z}}=\dfrac{\Delta_{g_0}}{4}\] is globally well-defined, where $\Delta_{g_0}$ denotes the Laplace--Beltrami operator with respect to the metric $g_0$. And also $|q|^2h_0^3$ is globally well-defined, which is the norm square $|q|_{g_0}^2$ of $q$ with respect to the metric on $\KK_X^{3/2}$ induced by the metric $g_0$ on $\KK_X^{1/2}$.
Thus $u_t$ satisfies the following equation
\begin{equation}\label{eq:conformalHitchin}
    \Delta_{g_0}u_t=4(\mathrm{e}^{2u_t}-1-t^2\mathrm{e}^{-u_t}\cdot|q|_{g_0}^2).
\end{equation}

Also, using the curvature formula for conformal metrics, \prettyref{eq:localHit} yields that
\[K(g_t)=-4(1-t^2|q|_{g_t}^2)\geqslant-4,\]
where $K$ denotes the sectional curvature. When $t=0$, this gives us $K(g_0)=-4$ which is a constant. Therefore, $g_X=4g_0$.

Below we want to consider the analyticity of $u_t$ with respect to $t$. By \prettyref{prop:stability}, $(\EE,\Phi_t,\mathcal{B})$ is stable. The analyticity actually follows from the real-analyticity of the non-Abelian Hodge correspondence. For completeness, we give a proof of the analyticity directly from the equation \prettyref{eq:conformalHitchin}.

\begin{lemma}\label{lem:analycityofmetric}
    $u_t$ depends on $t$ analytically.
\end{lemma}

\begin{proof}
     When $t\equiv0$, we consider the following operator
    \[\begin{aligned}
        \mathcal F\colon\mathbb R\times\mathcal{C}^{2+\alpha}(X)&\longrightarrow\mathcal{C}^{\alpha}(X),\\
        (t,u)&\longmapsto\Delta_{g_0}u-4(\mathrm{e}^{2u}-1-t^2|q|_{g_0}^2\mathrm{e}^{-u}),
    \end{aligned}\]
where $\mathcal{C}^\beta(X)$ denotes the space of $\beta$-H\"older functions on $X$. Then $\mathcal{F}$ is real analytic and $\mathcal F(t,u_t)=0$. Then for $v\in \mathcal{C}^{2+\alpha}(X)$, the $u$-linearized operator $D_u\mathcal F$ is

\[(D_u\mathcal F)_{(t,u)}(v) = \left.\frac{\dd}{\dd s}\right|_{s=0}\mathcal F(t,u+sv)=
\Delta_{g_0}v -\left(8\mathrm{e}^{2u}+4t^2|q|_{g_0}^2\mathrm{e}^{-u}\right)v.
\]

Because $8\mathrm{e}^{2u}+4t^2|q|_{g_0}^2\mathrm{e}^{-u}>0$, $D_u\mathcal F$ has trivial kernel by the maximum principle. Thus $D_u\mathcal F$ is actually invertible at every solution $(t,u)$ since $\Delta_{g_0}$ has trivial index. Therefore the analytic implicit function theorem  for Banach spaces gives a unique real-analytic map
\[
t\longmapsto u_t\in \mathcal{C}^{2+\alpha}(X).
\]
\end{proof}

As a corollary, we obtain that
\begin{corollary}
    The harmonic metric $H_t$, the Chern connection $\nabla_t$, the flat connection $\mathrm{D}_t$ and the representation $\rho_t$ all depend on $t$ analytically.
\end{corollary}

\subsection{Variations of the conformal coefficient}

\begin{lemma}\label{lem:variation-conf}
    \[\dot{u}_0=\dfrac{\dd u_t}{\dd t}\bigg|_{t=0}=0,\]
    and there exists a positive constant $m=m(q)>0$ such that \[\ddot{u}_0=\dfrac{\dd^2 u_t}{\dd t^2}\bigg|_{t=0}>m.\]
    As a corollary, \[\dfrac{\dd h_t}{\dd t}\bigg|_{t=0}=0,\quad\dfrac{\dd h_t^{-1}}{\dd t}\bigg|_{t=0}=0.\]
\end{lemma}

\begin{proof}
    Take $\dd/\dd t$ on \prettyref{eq:conformalHitchin}, we obtain that
\begin{equation}\label{eq:first-derivative}
    \Delta_{g_0}\dot{u}_t=4(2\dot{u}_t\mathrm{e}^{2u_t}-2t\mathrm{e}^{-u_t}\cdot|q|_{g_0}^2+t^2\dot{u}_t\mathrm{e}^{-u_t}\cdot|q|_{g_0}^2).
\end{equation}
When $t=0$, substitute $u_0=0$ into \prettyref{eq:first-derivative} implies that
\[(\Delta_{g_0}-8)\dot{u}_0=0.\]
Therefore, $\dot{u}_0=0$ by the maximum principle.

Now take the derivative $\dd/\dd t$ on \prettyref{eq:first-derivative}, we obtain that
\begin{equation}\label{eq:second-derivative}
    \begin{aligned}
        &\Delta_{g_0}\ddot{u}_t\\
        =&4\left(2\ddot{u}_t\mathrm{e}^{2u_t}+4(\dot{u}_t)^2\mathrm{e}^{2u_t}-2\mathrm{e}^{-u_t}\cdot|q|_{g_0}^2+2t\dot{u}_t\mathrm{e}^{-u_t}\cdot|q|_{g_0}^2\right.\\
    &\left.+2t\dot{u}_t\mathrm{e}^{-u_t}\cdot|q|_{g_0}^2+t^2\ddot{u}_t\mathrm{e}^{-u_t}\cdot|q|_{g_0}^2-t^2(\dot{u}_t)^2\mathrm{e}^{-u_t}\cdot|q|_{g_0}^2\right).
    \end{aligned}
\end{equation}
When $t=0$, substitute $u_0=0$ and $\dot{u}_0=0$ into \prettyref{eq:second-derivative} implies that
\[(\Delta_{g_0}-8)\ddot{u}_0=-8|q|_{g_0}^2.\]
Let $x_0$ be the minima of $\ddot{u}_0$, which exists since $X$ is compact. Then 
\[8\ddot{u}_0(x_0)\geqslant(8-\Delta_{g_0})\ddot{u}_0(x_0)=8|q|_{g_0}^2(x_0)\geqslant0,\]
and the strong maximum principle yields that if the equality holds, then $|q|_{g_0}^2\equiv0$, which contradicts to $q\neq0$. Therefore, $\ddot{u}_0>0$ and the compactness of $X$ implies that there exists a positive constant $m=m(q)$ such that $\ddot{u}_0>m$.
\end{proof}

\section{Variations of the eigenvalues}\label{sec: variations of eigengap}
We use a dot (e.g., $\dot{f}$) to denote the derivative with respect to the parameter $t$, and a prime (e.g., $f'$) to denote the derivative with respect to the parameter $s$.

Recall that we fix a nonzero $q\in\mathrm{H}^0(X,\KK_X^{3/2})$ and define the Higgs bundle $(\EE,\Phi_t,\mathcal{B})$ with harmonic metric $H_t$. Also we have the corresponding flat connection
\[\mathrm{D}_t=\nabla_t+\Phi_t+\Phi_t^{*_{H_t}}\]
and its holonomy $\rho_t$.

 We denote $a_i(t,\gamma):=\log|\lambda_i(\rho_t(\gamma))|$, where $i=1,2,3$ and $\gamma\in\pi_1(S)$. When $t=0$, since $\rho_0=j_X\oplus\mathbf{1}$, 
\[a_1(0,\gamma)=\dfrac{\ell_X(\gamma)}{2},\quad a_2(0,\gamma)=0,\quad a_3(0,\gamma)=-\dfrac{\ell_X(\gamma)}{2}.\] 
The goal of this section is to prove the following theorem:
\begin{theorem}
    \label{thm:variation-flat}
    \[\dot a_1(0,\gamma)=0,\quad \dot a_3(0,\gamma)=0\]
    and there exists a positive constant $m=m(q)>0$ such that
    \[\ddot a_1(0,\gamma)\geqslant \dfrac{m\cdot\ell_X(\gamma)}{2},\quad \ddot a_3(0,\gamma)\leqslant -\dfrac{m\cdot\ell_X(\gamma)}{2},\] for any $\gamma\in\pi_1(S)$.
\end{theorem}

\subsection{The initial flat connection along a closed geodesic}

Fix a nontrivial $\gamma\in\pi_1(S)$. Let $c\colon[0,T_\gamma]\to S$ be the unit-speed $g_0$-geodesic corresponding to $\gamma$. Then $T_\gamma=\ell_X(\gamma)/2$.

Let $\mathbf{e}=(e_+,e_0,e_-)$ be a $\nabla^{H_0}$-flat frame along $c$ which is adapted to $\KK_X^{1/2}\oplus\mathcal{O}_X\oplus\KK_X^{-1/2}$ and $H_0$-unitary. With respect to this frame, we can express
\[c^*\mathrm{D}_0=\dd+A_0\dd s,\]
where $c^*\mathrm{D}_0$ is the connection on the pulled back bundle, $s$ is the arc-parameter on $c$. Let $c'=\dd c/\dd s$ denotes the tangent vector field of $c$. Then 
\[A_0=\Phi_0(c')+\Phi_0^{*_{H_0}}(c')=\Phi_0((c')^{1,0})+\Phi_0^{*_{H_0}}((c')^{0,1})\]
by definition.

The following proposition is essentially proven in \cite[Proposition 4.1]{bronstein2025anosov}. We give a proof without using the notion of parallel distribution. Let
$\nabla^{\End}_0$ be the induced connection on $\End(\EE)$. 

\begin{proposition}\label{prop:parallelism}
For real vector fields $U,Y$ on $X$, define
\[
V(Y):=\Phi_0(Y^{1,0})+\Phi_0^{*_{H_0}}(Y^{0,1}).\]
Then
\[
\nabla^{\End}_{0,U}V(Y)=V(\nabla^{g_0}_U Y).
\]
In particular, if $c$ is a $g_0$-geodesic, then
\[
\nabla^{\End}_{0,c'}V(c')=0.
\]
\end{proposition}

\begin{proof}
Choose a holomorphic coordinate $z$ and a local holomorphic frame $s$ of
$\KK_X^{1/2}$ with $s^2=\dd z$. In the frame $(s,1,s^{-1})$ write
$|s|_{H_0}^2=h_0$.
Then $g_0=h_0^{-2}|\dd z|^2$.

Let $E_{ij}$ denote the elementary matrix. Then $\Phi_0=E_{31}\dd z$ and
$\Phi_0^{*_{H_0}}=h_0^{-2}E_{13}\dd\bar z$.
For
\[
Y=y\,\partial_z+\bar y\,\partial_{\bar z},
\]
we therefore have
\begin{equation}\label{eq:Vlocal}
V(Y)=yE_{31}+h_0^{-2}\bar yE_{13}.
\end{equation}

The Chern connection $\nabla_0$ is
\[
\nabla_0=\dd+\operatorname{diag}(\partial\log h_0,0,-\partial\log h_0).
\]
Hence
\[
\nabla_0^{\End}E_{31}=-2\partial\log h_0\,E_{31},
\qquad
\nabla_0^{\End}E_{13}=2\partial\log h_0\,E_{13}.
\]
If
\[
U=u\partial_z+\bar u\partial_{\bar z},
\]
then differentiating \eqref{eq:Vlocal} gives
\begin{equation}\label{eq:Vderivative}
\nabla^{\End}_{0,U}V(Y)
=A E_{31}+h_0^{-2}\bar A E_{13},
\end{equation}
where
\begin{equation}\label{eq:Acoef}
A=U(y)-2u(\partial_z\log h_0)y.
\end{equation}
Indeed, in the $E_{13}$ term the derivative of $h_0^{-2}$ cancels the
$2\partial\log h_0$ contribution from the induced Chern connection.

For the conformal metric $g_0=h_0^{-2}|dz|^2$, the Christoffel symbols are
\[
\Gamma^z_{zz}=\partial_z\log(h_0^{-2})=-2\partial_z\log h_0,
\qquad
\Gamma^z_{\bar z z}=0.
\]
Consequently
\[
(\nabla_U^{g_0}Y)^{1,0}=A\,\partial_z.
\]
Applying \eqref{eq:Vlocal} to $\nabla_U^{g_0}Y$ gives exactly the right-hand
side of \eqref{eq:Vderivative}. 
The geodesic case follows from the fact that $\nabla^{g_0}_{c'}c'=0$.
\end{proof}

Therefore, by \prettyref{prop:parallelism}, 
\[A_0(\mathbf{e})=V(c')(\mathbf{e})=\mathbf{e}\cdot\begin{pmatrix}
    0&0&\bar\eta\\0&0&0\\\eta&0&0
\end{pmatrix}\]
for a constant $\eta\in\mathbb{C}^*$. Since $c$ is of unit-speed, $|\eta|^2=1$. Replacing $e_-$ by $\eta e_-$, we can assume that $\eta=1$ and thus
\[A_0(\mathbf{e})=\mathbf{e}\cdot\begin{pmatrix}
    0&0&1\\0&0&0\\1&0&0
\end{pmatrix}\]
for the frame $\mathbf{e}=(e_+,e_0,e_-)$. We can diagonalize it by another $\nabla^{H_0}$-flat and $H_0$-unitary frame $\mathbf{u}:=(u_+,u_0,u_-)$, where
\[u_+:=\dfrac{e_+-e_-}{\sqrt{2}},\quad u_0:=e_0,\quad u_-:=\dfrac{e_++e_-}{\sqrt{2}}.\]
For instance, \begin{equation}\label{eq:A0}
    A_0(\mathbf{u})=\mathbf{u}\cdot\operatorname{diag}(-1,0,1).
\end{equation}

\subsection{The first variation}
Below we write the Taylor expansion
\[c^*\mathrm{D}_t=\dd+A_t\dd s=\dd+\left(A_0+tB+\dfrac{t^2}{2}C+O(|t|^3)\right)\dd s,\]
where $B=\dot{A}_0$, $C=\ddot{A}_0$.

Note that the Chern connection can be written as
\[\nabla_t=\dd+H_t^{-1}\partial H_t.\]
Thus
\[\dot{\nabla}_0=\dfrac{\dd\nabla_t}{\dd t}\bigg|_{t=0}=-H_0^{-1}\dot{H}_0H_0^{-1}\partial H_0+H_0^{-1}\partial\dot{H}_0=0\]
by $\dot{H}_0=0$, which is obtained from \prettyref{lem:variation-conf}. 

The $\mathds{1}$-part of the Higgs field is independent of $t$, and its adjoint has zero first derivative because $\dot{H}_0=0$.  Therefore, $B$ comes only from the linear $q$-term and its adjoint. Evaluated on a real tangent
vector this is Hermitian. Thus
\[B(\mathbf{e})=\mathbf{e}\cdot\begin{pmatrix}
    0&Q&0\\\overline Q&0&Q^\star\\0&\overline{Q^\star}&0
\end{pmatrix},\]
where $Q,Q^\star$ are complex-valued functions with $Q^\star=\eta^\star Q$ for $\eta^\star \in \mathbb C,\  |\eta^\star | = 1$ since the frame $\mathbf{e}$ is $H_0$-unitary and $\nabla_0$-flat. Therefore, by replacing $e_0$ by $\sqrt{\eta^\star}e_0$, we can assume that $Q=Q^\star$. Let $Q=Q_1+\iu Q_2$ for real-valued functions $Q_1,Q_2$. Then
\begin{equation}\label{eq:Bmatrix}
    B(\mathbf{u})=\mathbf{u}\cdot\begin{pmatrix}
    0&\sqrt{2}\iu Q_2&0\\-\sqrt{2}\iu Q_2&0&\sqrt{2}Q_1\\0&\sqrt{2}Q_1&0
\end{pmatrix}.
\end{equation}

\subsection{The second variation}

We decompose $C$ as
\[
C=C_{\rm Ch}+C_{\mathds{1}}+C_q,
\]
where $C_{\rm Ch}$ arises from the Chern connection $\nabla_t$, $C_{\mathds{1}}$ arises from the $\mathds{1}$-part of the Higgs field and its adjoint, and $C_q$ arises from the linear $q$-term of the Higgs field and its adjoint.

\subsubsection*{The Chern term}

Recall that the harmonic metric $H_t=\operatorname{diag}(h_t,1,h_t^{-1})$ with respect to the decomposition $\EE=\KK_X^{1/2}\oplus\mathcal{O}_X\oplus\KK_X^{-1/2}$. 
By definition,
\[
\log h_t=\log h_0-u_t.
\]
Therefore,
\[
C_{\rm Ch}(\mathbf{e})=\ddot{\nabla}_0(c')(\mathbf{e})=
\mathbf{e}\cdot\begin{pmatrix}
-\alpha&0&0\\
0&0&0\\
0&0&\alpha
\end{pmatrix}
\]
for a complex function $\alpha$ proportional to
$\partial\ddot{u}_0(c'^{1,0})$. 
Hence
\begin{equation}\label{eq:Cchern-diag}
C_{\rm Ch}(\mathbf{u})=\ddot{\nabla}_0(c')(\mathbf{u})=
\mathbf{u}\cdot\begin{pmatrix}
0&0&-\alpha\\
0&0&0\\
-\alpha&0&0
\end{pmatrix}.
\end{equation}

\subsubsection*{The \texorpdfstring{$q$}{q}-term}

Write $\Phi_t=\Phi_0+t\Psi$.  The holomorphic term $t\Psi$ is
exactly linear. Moreover
\[
H_t=H_0+O(|t|^2)
\]
by \prettyref{lem:variation-conf},
so
\[
t\Psi^{*_{H_t}}=t\Psi^{*_{H_0}}+O(|t|^3).
\]
There is no $t^2$ contribution.  Hence
\begin{equation}\label{eq:Cqzero}
C_q=0.
\end{equation}

\subsubsection*{The \texorpdfstring{$\mathds{1}$}{1}-term}

Since the harmonic metric $H_t=\operatorname{diag}(h_t,1,h_t^{-1})$ with respect to the decomposition $\EE=\KK_X^{1/2}\oplus\mathcal{O}_X\oplus\KK_X^{-1/2}$, a direct computation shows that
\[
C_{\mathds{1}}(\mathbf{e})=\mathbf{e}\cdot\left.\dfrac{\dd^2}{\dd t^2}\right|_{t=0}\begin{pmatrix}
0&0&\mathrm{e}^{2u_t}\\0&0&0\\1&0&0
\end{pmatrix}=\mathbf{e}\cdot\begin{pmatrix}
0&0&2\ddot{u}_0\\0&0&0\\0&0&0
\end{pmatrix}.
\]
Therefore, 
\begin{equation}\label{eq:Ctauoriginal}
C_{\mathds{1}}(\mathbf{u})=\mathbf{u}\cdot
\begin{pmatrix}
-\ddot{u}_0&0&\ddot{u}_0\\
0&0&0\\
-\ddot{u}_0&0&\ddot{u}_0
\end{pmatrix}.
\end{equation}
Consequently, by \prettyref{eq:Cchern-diag}, \prettyref{eq:Cqzero} and \prettyref{eq:Ctauoriginal}, we obtain that
\begin{equation}\label{eq:Cdiagexact}
C(\mathbf{u})=\mathbf{u}\cdot\begin{pmatrix}
    -\ddot{u}_0&0&\ddot{u}_0-\alpha\\
0&0&0\\
-\ddot{u}_0-\alpha&0&\ddot{u}_0
\end{pmatrix}.
\end{equation}

\subsection{Labourie--Wentworth's formula}

To compute the variations of the eigenvalues, we will use the following formula proven by Labourie--Wentworth in \cite[Lemma 4.1.1]{labourie2018variations}.

\begin{lemma}\label{lem:LW}
    Consider a connection $\nabla$ and a closed curve $\gamma$ in $S$ so that the holonomy of $\nabla$ along $\gamma$ has an eigenvalue of multiplicity $1$. We denote by $L_\gamma$ the corresponding eigenline along $\gamma$, by $H_\gamma$ the supplementary hyperplane stable by the holonomy, and by $P_\gamma$ the projection on $L_\gamma$ along $H_\gamma$. 
    
    Let $\nabla_t$ be a smooth one parameter family of connections with $\nabla_0=\nabla$. Then there exists (for $t$ small enough) a unique smooth function $\lambda_\gamma(t)$ such that 
        \begin{itemize}
            \item $\lambda_\gamma(0)=\lambda_\gamma$;
            \item $\lambda_\gamma(t)$ is an eigenvalue of the holonomy of $\nabla_t$ of multiplicity $1$.
        \end{itemize}
 Moreover,
    \[\dot{\lambda}_\gamma(0)=\left.\dfrac{\dd}{\dd t}\right|_{t=0}\lambda_\gamma(t)=-\lambda_\gamma\cdot\int_{0}^{\ell(\gamma)}\operatorname{tr}(\dot{\nabla}\cdot P_\gamma)\dd s,\]
    where $\ell(\gamma)$ denotes the length of $\gamma$ and $s$ denotes the arc parameter.
\end{lemma}

In our case, recall that \prettyref{eq:A0} yields
\[(c^*\mathrm{D}_0)(\mathbf{u})=\mathbf{u}\cdot\operatorname{diag}(-1,0,1).\]
Hence a basis of the space of $\mathrm{D}_0$-flat sections along $c$ is given by $(\mathrm{e}^{s}u_+, u_0, \mathrm{e}^{-s}u_-)$. They are eigenvectors of the holonomy with eigenvalue modulus $\mathrm{e}^{T_\gamma},1,\mathrm{e}^{-{T_\gamma}}$ respectively, where $T_\gamma=\ell_X(\gamma)/2$. For small $t$, let
\[
M_t=\operatorname{Hol}(\mathrm{D}_t)(c):\EE_{c(0)}\to \EE_{c(0)}
\]
be the holonomy around the closed geodesic $c$. Its three eigenvalues are simple by \prettyref{lem:LW} (This is also known for arbitrary $t\in\mathbb{R}$ by \cite[Theorem 1.1]{bronstein2025anosov}).
At the initial point choose the spectral projection
\[
P_i(t,0)\colon\EE_{c(0)}\to \EE_{c(0)}
\]
onto the $i$-th eigenline along the direct sum of the other two eigenspaces.
If $U_t(s)$ denotes $\mathrm{D}_t$-parallel transport from $0$ to $s$, define
\[
P_i(t,s)=U_t(s)P_i(t,0)U_t(s)^{-1}.
\]
Because $P_i(t,0)$ is a spectral projection of $M_t$, it commutes with
$M_t$. Hence $P_i(t,s)$ descends to a well-defined section $P_i(t)$ of
$\End(\EE)$ over the closed geodesic $c$. Note that
$P_i(\mathbf{u})=\mathbf{u}\cdot E_{ii}$, where $E_{ii}$ denotes the elementary matrix. Here and below $P_i=P_i(0)$ and $\dot{P}_i =\dot P_i(0)$, $i = 1,2,3$.

\begin{proposition}\label{prop:first-der}
When $t$ is sufficiently small, 
\begin{equation}\label{eq:ourLW}
    \begin{cases}
    \dot a_1(t,\gamma)=-\operatorname{Re}\displaystyle\int_0^{T_\gamma}\operatorname{tr}(\dot{A}_t\cdot P_1(t))\dd s,\\
    \dot a_3(t,\gamma)=-\operatorname{Re}\displaystyle\int_0^{T_\gamma}\operatorname{tr}(\dot{A}_t\cdot P_3(t))\dd s.
    \end{cases}
\end{equation}
 Moreover, $\dot a_1(0,\gamma)=0$ and $\dot a_3(0,\gamma)=0$.
\end{proposition}

\begin{proof}
    Applying \prettyref{lem:LW} to our case, we obtain \prettyref{eq:ourLW} by definition. Then
    \[\begin{aligned}
        \dot a_1(0,\gamma)
        =&-\operatorname{Re}\int_0^{T_\gamma}\operatorname{tr}(B\cdot P_1)\dd s\quad(\text{by \prettyref{eq:ourLW}})\\
        =&-\operatorname{Re}\int_0^{T_\gamma}\operatorname{tr}\left(\begin{pmatrix}
            0&\sqrt{2}\iu Q_2&0\\-\sqrt{2}\iu Q_2&0&\sqrt{2}Q_1\\0&\sqrt{2}Q_1&0\end{pmatrix}\cdot\begin{pmatrix}
            1&0&0\\0&0&0\\0&0&0
        \end{pmatrix}\right)\dd s\quad(\text{by \prettyref{eq:Bmatrix}})\\
        =&-\operatorname{Re}\int_0^{T_\gamma}0\ \dd s=0.
    \end{aligned}\]
    Similarly, we obtain that $\dot a_3(0,\gamma)=0$.
\end{proof}

To compute the second order derivatives, we differentiate \prettyref{eq:ourLW}, then
\begin{lemma}\label{lem:second-der}
    \[\begin{cases}
    \ddot a_1(0,\gamma)=-\operatorname{Re}\displaystyle\int_0^{T_\gamma}
\operatorname{tr}\left(C\cdot P_1+B\cdot\dot P_1\right)\dd s,\\
    \ddot a_3(0,\gamma)=-\operatorname{Re}\displaystyle\int_0^{T_\gamma}
\operatorname{tr}\left(C\cdot P_3+B\cdot\dot P_3\right)\dd s.
    \end{cases}\]
\end{lemma}
This forces us to compute the first variation of the spectral projections.

\subsection{The first variation of the spectral projections}

The connection $\mathrm{D}_t$ induces a connection on $\End(\mathcal E)$ by
\[
\mathrm{D}_t^{\End}P=\mathrm{D}_t\circ P-P\circ \mathrm{D}_t.
\]
By construction, along the closed curve we have
\[
\mathrm{D}_t^{\End}P_i(t)=0.
\]
Therefore, using the frame expression $c^*\mathrm{D}_t=\dd+A_t\dd s$, we obtain that
\[
P_i'(t)+[A_t,P_i(t)]=0.
\]
Differentiating at $t=0$ gives
\begin{equation}\label{eq:Pvariation}
\dot P_i'+[A_0,\dot P_i]+[B,P_i]=0.
\end{equation}

Also, differentiating $P_i(t)^2=P_i(t)$ with respect to $t$ gives
\begin{equation}\label{eq:idempotentvar}
P_i\dot P_i+\dot P_iP_i=\dot P_i.
\end{equation}
Thus $\dot P_i$ has only off-diagonal blocks between the image of $P_i$ and
its complementary plane.

\begin{lemma}\label{lem:variation-proj}
    There exists complex-valued functions $x,y,v,w$ on $[0,{T_\gamma}]$ such that \[\dot{P}_1(\mathbf{u})=\mathbf{u}\cdot\begin{pmatrix}
        0&x&0\\y&0&0\\0&0&0
    \end{pmatrix},\quad x'-x=y'+y=\sqrt{2}\iu Q_2\] and
    \[\dot{P}_3(\mathbf{u})=\mathbf{u}\cdot\begin{pmatrix}
        0&0&0\\0&0&v\\0&w&0
    \end{pmatrix},\quad -(v'-v)=w'+w=\sqrt{2} Q_1,\]
    where $Q_1,Q_2$ are functions in \prettyref{eq:Bmatrix}.
\end{lemma}

\begin{proof}
    We only prove the formula for $\dot P_1$. The one for $\dot{P}_3$ is similar.

    By \prettyref{eq:idempotentvar}, 
    \[\dot{P}_1(\mathbf{u})=\mathbf{u}\cdot\begin{pmatrix}
        0&x&p_{13}\\y&0&0\\p_{31}&0&0
    \end{pmatrix}.\]
    Now \prettyref{eq:Pvariation} turns to be
    \[\begin{aligned}
        &\begin{pmatrix}
        0&x'&p_{13}'\\y'&0&0\\p_{31}'&0&0
    \end{pmatrix}+\left[\begin{pmatrix}
        -1&0&0\\0&0&0\\0&0&1
    \end{pmatrix},\begin{pmatrix}
        0&x&p_{13}\\y&0&0\\p_{31}&0&0
    \end{pmatrix}\right]\\
    +&\left[\begin{pmatrix}
    0&\sqrt{2}\iu Q_2&0\\-\sqrt{2}\iu Q_2&0&\sqrt{2}Q_1\\0&\sqrt{2}Q_1&0
\end{pmatrix},\begin{pmatrix}
        1&0&0\\0&0&0\\0&0&0
    \end{pmatrix}\right]=0
    \end{aligned}\]
    by \prettyref{eq:A0} and \prettyref{eq:Bmatrix}. This is equivalent to
    \[\begin{cases}
        x'-x=y'+y=\sqrt{2}\iu Q_2,\\
        p_{13}'-2p_{13}=0,\\
        p_{31}'+2p_{31}=0.
    \end{cases}\]
    
    Since the $\nabla_0$-flat frame $\mathbf{u}$ is $H_0$-unitary, the moduli $|p_{13}({T_\gamma})|=|p_{13}(0)|$, $|p_{31}({T_\gamma})|=|p_{31}(0)|$. Thus the only possible function satisfying that 
    \[p_{13}'-2p_{13}=0,\quad|p_{13}({T_\gamma})|=|p_{13}(0)|\] is zero. Similarly, $p_{31}\equiv0$. 
\end{proof}

\subsection{Proof of \prettyref{thm:variation-flat}}

In this subsection, we prove \prettyref{thm:variation-flat}.

\begin{proof}[Proof of \prettyref{thm:variation-flat}]
    We have proven the vanishing of the first order derivatives in \prettyref{prop:first-der}. By \prettyref{lem:second-der},
    \[\ddot a_1(0,\gamma)=-\operatorname{Re}\int_0^{T_\gamma}
\operatorname{tr}\left(C\cdot P_1+B\cdot\dot P_1\right)\dd s.\]

Firstly,
\[\begin{aligned}
    &-\operatorname{Re}\int_0^{T_\gamma}
\operatorname{tr}\left(C\cdot P_1\right)\dd s\\
=&-\operatorname{Re}\int_0^{T_\gamma}
\operatorname{tr}\left(\begin{pmatrix}
    -\ddot{u}_0&0&\ddot{u}_0-\alpha\\
0&0&0\\
-\ddot{u}_0-\alpha&0&\ddot{u}_0
\end{pmatrix}\cdot\begin{pmatrix}
            1&0&0\\0&0&0\\0&0&0
        \end{pmatrix}\right)\dd s\quad(\text{by \prettyref{eq:Cdiagexact}})\\
        =&\int_{0}^{T_\gamma}\ddot{u}_0\dd s.
\end{aligned}\]
Then
\[\begin{aligned}
    &-\operatorname{Re}\int_0^{T_\gamma}
\operatorname{tr}\left(B\cdot\dot P_1\right)\dd s\\
=&-\operatorname{Re}\int_0^{T_\gamma}
\operatorname{tr}\left(\begin{pmatrix}
    0&\sqrt{2}\iu Q_2&0\\-\sqrt{2}\iu Q_2&0&\sqrt{2}Q_1\\0&\sqrt{2}Q_1&0
\end{pmatrix}\cdot\begin{pmatrix}
            0&x&0\\y&0&0\\0&0&0
        \end{pmatrix}\right)\dd s\quad(\text{by \prettyref{eq:Bmatrix} and \prettyref{lem:variation-proj}})\\
        =&-\operatorname{Re}\int_0^{T_\gamma}\sqrt{2}\iu Q_2(y-x)\dd s\\
        =&\operatorname{Re}\int_0^{T_\gamma}\left(y\cdot\overline{y'+y}-x\cdot\overline{x'-x}\right)\dd s\quad(\text{by \prettyref{lem:variation-proj}})\\
        =&\int_0^{T_\gamma}\left(|y|^2+|x|^2\right)\dd s+\operatorname{Re}\int_0^{T_\gamma}\left(y\cdot\overline{y'}-x\cdot\overline{x'}\right)\dd s\\
        \geqslant&\operatorname{Re}\int_0^{T_\gamma}\left(y\cdot\overline{y'}-x\cdot\overline{x'}\right)\dd s.
\end{aligned}\]
       Note that
       \[
        \operatorname{Re}\int_0^{T_\gamma}y\cdot\overline{y'}\dd s=\dfrac{1}{2}\cdot\int_0^{T_\gamma}\left(y\cdot\overline{y'}+y'\cdot\overline{y}\right)\dd s=\dfrac{1}{2}\cdot\int_0^{T_\gamma}\dfrac{\dd|y|^2}{\dd s}\dd s=\dfrac{1}{2}\cdot \left(|y({T_\gamma})|^2-|y(0)|^2\right)=0,\]
        where the last equality follows from the fact that the frame $\mathbf{u}$ is $H_0$-unitary. Similarly, we can show that
        \[\operatorname{Re}\int_0^{T_\gamma}\left(y\cdot\overline{y'}-x\cdot\overline{x'}\right)\dd s=0.\]

        Combining all inequalities above, we obtain that
        \[\ddot a_1(0,\gamma)\geqslant\int_0^{T_\gamma}\ddot{u}_0\ \dd s.\] 
        By \prettyref{lem:variation-conf} and ${T_\gamma}=\ell_X(\gamma)/2$, and similar computation on $\ddot a_3(0,\gamma)$, we obtain the desired result.
\end{proof}

\section{Proof of \prettyref{thm:Hilbert domination}}
In this section, we prove \prettyref{thm:Hilbert domination} using \prettyref{thm:variation-flat}. To be more precise, in the remainder of this paper we adopt the following conventions for deformation parameters:
\begin{enumerate}
    \item The symbol $t\in \mathbb{R}$ is only used to parameterize the family of representations $\rho_{t}:=\operatorname{NAH}(\EE,\Phi_{tq})$ defined above and its derived objects.
    \item The symbol $u\in \mathbb{R}^k$ is only used to parameterize general analytic families.
    \item The symbol $\nu\in \mathbb{R}$ is only used in \prettyref{rem: dimension>3/2} to parameterize the deformation there and its derived objects.
\end{enumerate}

Let $T^1S$ be the unit tangent bundle of $S$, equipped with the geodesic flow $\varphi^\tau, \tau \in \mathbb{R}$, and let $\iota:T^1S\to T^1S$ be the flip map $(x,v)\mapsto(x,-v)$. For each nontrivial $\gamma\in\pi_1(S)$, let $\ell_\gamma\subset T^1S$ be the corresponding periodic orbit, oriented so that its lift to the covering $T^1\mathbb{H}$ points from the repelling fixed point $\gamma^-$ to the attracting fixed point $\gamma^+$. For a continuous function $f$ on $T^1 S$, the integral $\int_{\ell_\gamma}f\dd \tau$ is taken with the multiplicity of $\gamma$. With this convention, $\ell_{\gamma^{-1}}=\iota(\ell_\gamma)$. Note that for any non-trival $\gamma\in \pi_1(S)$, $\int_{\ell_{\gamma}} 1 \dd \tau = \ell_X(\gamma)$. For each $\alpha>0$, let $\mathcal{C}^{\alpha}(T^1S)$ be the Banach space of $\alpha$-H\"{o}lder functions on $T^1S$.

We recall a theorem from the thermodynamic formalism, which shows that logarithms of eigenvalue moduli can be realized as integrals over periodic orbits. Let $D\subset \mathbb{R}^k$ be an open set containing $0$, and let $\rho_u: \pi_1(S)\to \mathrm{SL}_3 \mathbb{R},\ u\in D$ be a family of Anosov representations varying analytically in $\mathrm{Hom}(\pi_1(S), \mathrm{SL}_3 \mathbb{R})$. 

\begin{theorem}[{Bridgeman--Canary--Labourie--Sambarino~\cite[Proposition~6.2]{BCLS}}]\label{thm:bcls-variation}
There exist an open neighbourhood $U\subset D$ of $0$, $\alpha>0$, and a family of positive $\alpha$-H\"{o}lder functions $\{f_u\}_{u\in U}$ on $T^1S$ such that for any nontrivial $\gamma\in \pi_1(S)$,
\begin{equation}\label{eq: bclsintegralformula}
\log |\lambda_1(\rho_u(\gamma))| = \int_{\ell_{\gamma}} f_u \dd \tau, \quad \log |\lambda_3(\rho_u(\gamma))| = -\int_{\ell_{\gamma}} (f_u\circ\iota) \dd \tau.    
\end{equation}

Moreover, the map $u\mapsto f_u$ is real analytic from $U$ to $\mathcal{C}^{\alpha}(T^1S)$.
\end{theorem}

\begin{remark}
In the original setting of \cite{BCLS}, it was established that for any $1$-Anosov representation $\rho:\pi_1(S)\to \mathrm{SL}_n\mathbb{R}$, one can find a positive H\"{o}lder function $f$ on $T^1S$, depending analytically on $\rho$, such that for any nontrivial $\gamma\in \pi_1(S)$,
\[
\log|\lambda_1(\rho(\gamma))| = \int_{\ell_\gamma} f\dd\tau.
\]
The desired formulation then follows by combining
\[
\int_{\ell_{\gamma}} f \dd \tau  = \int _{\ell_{\gamma^{-1}}} (f\circ \iota) \dd \tau
\]
with the identity $\left|\lambda_1(g^{-1})\right| = \left|\lambda_3(g)\right|^{-1}$ for any $g\in \mathrm{SL}_3\mathbb{R}$.
\end{remark}

\begin{proof}[Proof of \prettyref{thm:Hilbert domination}]
Again from the identity $\left|\lambda_1(g^{-1})\right| = \left|\lambda_3(g)\right|^{-1}$ for any $g\in \mathrm{SL}_3\mathbb{R}$, it suffices to prove the inequality for $\log \left |\lambda_1(\rho_t(\gamma))\right|$.

By \prettyref{thm:bcls-variation}, there exist $\varepsilon'>0$, $\alpha>0$, and an analytic family of positive functions $\{f_t\}_{t\in (-\varepsilon', \varepsilon')}\subset \mathcal{C}^{\alpha}(T^1S)$ such that for any $t\in (-\varepsilon',\varepsilon')$, \prettyref{eq: bclsintegralformula} holds. So there exist $\dot f, \ddot f\in \mathcal{C}^{\alpha}(T^1S)$ such that
\[
f_t= f_0+t\dot f+\frac{t^2}{2} \ddot f+ R_t,
\]
where the remainder $R_t \in \mathcal{C}^{\alpha}(T^1S)$ satisfies $\lim_{t\to 0} R_t/t^2 = 0$. Here the limit is taken in the topology of the Banach space $\mathcal{C}^{\alpha}(T^1S)$.

Following the notations introduced in \prettyref{sec: variations of eigengap}, for each nontrivial $\gamma\in \pi_1(S)$ and $t\in (-\varepsilon',\varepsilon')$, we have
\[
a_1(t,\gamma) = \int_{\ell_\gamma} f_0 \dd\tau+ \left(\int_{\ell_{\gamma}} \dot f\dd\tau \right) t+\left(\int_{\ell_{\gamma}} \ddot f\dd\tau\right)\frac{t^2}{2}+\int_{\ell_{\gamma}} R_t\dd\tau.
\]
Since $R_t/t^2$ converges to $0$ in the Banach space $\mathcal{C}^{\alpha}(T^1 S)$, it follows that
\[
\dot a_{1}(0,\gamma)  = \int_{\ell_{\gamma}} \dot f\dd\tau, \quad \ddot a_{1}(0,\gamma) = \int_{\ell_{\gamma}} \ddot f\dd\tau.
\]
\prettyref{thm:variation-flat} thus implies that for each nontrivial $\gamma\in \pi_1(S)$,
\begin{equation}\label{eq: potentialvarint}
\int_{\ell_{\gamma}} \dot f\dd\tau = 0,\quad \int_{\ell_{\gamma}} \ddot f\dd\tau\geqslant  \frac{m}{2}\cdot \ell_X(\gamma),     \end{equation}
where $m>0$ is a constant depending only on $q$.

Note that $a_{1}(0,\gamma) = \ell_{X}(\gamma)/2$. As a result, we have
\[
a_1(t,\gamma) \geqslant  \frac{1}{2}\ell_X(\gamma)+ \frac{m}{4} t^2 \ell_X(\gamma)+\int_{\ell_{\gamma}} R_t\dd\tau.
\]
Since $R_t/t^2 \to 0$ in $\mathcal{C}^{\alpha}(T^1 S)$, it also converges to $0$ in the $L^{\infty}$ norm. Thus, there exists $0<\varepsilon_0<\varepsilon'$ such that for all $t\in (-\varepsilon_0,\varepsilon_0)$, we have $\|R_t\|_{L^\infty} \leqslant mt^2/8$. This yields
\begin{equation}\label{eq: remainderint}
\left|\int_{\ell_{\gamma}} R_t \dd\tau\right|  \leqslant \frac{m}{8} t^2 \ell_X(\gamma), 
\end{equation}
which implies
\[
a_1(t,\gamma)\geqslant  \left(\frac{1}2+\frac{m}{8} t^2 \right)\ell_X(\gamma).
\]
Thus $m_0 = m/8$ is the desired constant depending only on $q$.
\end{proof}

\section{Applications}

In this section, we prove \prettyref{coro: 6dimnonhitchin}, \prettyref{coro:lyapunov exponent}, and \prettyref{coro: affinity exponent} as some applications of \prettyref{thm:Hilbert domination}.

\subsection{Non-Hitchin Borel Anosov representations}

\begin{proof}[Proof of \prettyref{coro: 6dimnonhitchin}]
    We first prove that $\rho_t\otimes \operatorname{Sym}^{k-1}(j_X)$ is Borel Anosov when $|t|$ is sufficiently small and nonzero. Keep using the notations in the proof of Theorem~\ref{thm:Hilbert domination}. For $t\in(-\varepsilon_0, \varepsilon_0)$, define the analytic families
    \begin{equation}\label{eq: formulasofgh}
    r_{1,t}: = 2 f_t - f_t\circ \iota ,\quad r_{2,t}: =2 f_t \circ \iota -f_t.        
    \end{equation}
    
    Since for any $g \in \mathrm{SL}_3\mathbb{R}$, $\log |\lambda_1 (g)|+\log|\lambda_2(g)|+\log |\lambda_3(g)| = 0$, we have for any nontrivial $\gamma \in \pi_1(S)$,
    \begin{equation}\label{eq: rootpotential}
    \log \frac{|\lambda_1(\rho_t(\gamma))|}{|\lambda_2(\rho_t(\gamma))|} = \int_{\ell_{\gamma}} r_{1,t} \dd\tau, \quad \log \frac{|\lambda_2(\rho_t(\gamma))|}{|\lambda_3(\rho_t(\gamma))|} = \int_{\ell_{\gamma}} r_{2,t} \dd\tau.
    \end{equation}
    Note that for any nontrivial $\gamma \in \pi_1(S)$,
    \[
    \int_{\ell_{\gamma}} r_{1,0} \dd \tau  = \int_{\ell_{\gamma}} r_{2,0} \dd \tau = \frac{1}{2}\ell_{X}(\gamma).
    \]
    Since $r_{1,t}$ and $r_{2,t}$ vary analytically in $\mathcal{C}^{\alpha}(T^1S)$, the differences $r_{1,t}-r_{1,0}$ and $r_{2,t}-r_{2,0}$ converge uniformly to $0$ as $t\to 0$. Thus, we can choose $0 < \varepsilon_1 < \min \left(\varepsilon_0, 1/(2\sqrt{2m_0})\right)$ such that for all $t\in(-\varepsilon_1, \varepsilon_1)$ and nontrivial $\gamma\in \pi_1(S)$,
    \begin{equation}\label{eq:smallsingularvaluevariation}
        \left|\int_{\ell_{\gamma}} r_{1,t} \dd \tau -\frac{1}{2}\ell_{X}(\gamma) \right| \leqslant \frac{1}{4} \ell_X(\gamma) \quad \text{and} \quad \left|\int_{\ell_{\gamma}} r_{2,t} \dd \tau -\frac{1}{2}\ell_{X}(\gamma) \right| \leqslant \frac{1}{4} \ell_X(\gamma).
    \end{equation}
    
    We first show that for any $t\in \mathbb{R}, |t|\in (0, \varepsilon_1)$, the representation $\left(\rho_t\otimes\operatorname{Sym}^{k-1}(j_X)\right)$ is Borel Anosov. Observe that for every $\gamma\in \pi_1(S)$, the eigenvalues of the tensor product representation $\left(\rho_t\otimes\operatorname{Sym}^{k-1}(j_X)\right)(\gamma)$ are given by
    \[
    \lambda_{i}(\rho_t(\gamma)) \cdot \lambda^{k-2j+1}(j_X(\gamma)),\quad i \in \{1,2,3\},\quad j \in \{1,\dots, k\}.
    \]
   Since $\rho_t$ is Borel Anosov, for any nontrivial element $\gamma\in \pi_1(S)$, its image $\rho_t(\gamma)$ is loxodromic (see Labourie~\cite[Proposition~3.4]{labourie2006anosov}). So all above eigenvalues are real. Consequently, the logarithms of all possible eigenvalue modulus gaps take the form
    \[
    \left| \log \left|\lambda_{i_1}(\rho_t(\gamma))\right|- \log \left|\lambda_{i_2}(\rho_t(\gamma))\right| -l \ell_X(\gamma) \right|,\quad  i_1,i_2\in \{1,2,3\},
    \]
    where $l \in \mathbb{Z}$ satisfies $1-k \leqslant l \leqslant k-1$, and $l\neq 0$ whenever $i_1 = i_2$. By \prettyref{defn:eigenvalue-gap}, it suffices to prove that for any $\gamma\in \pi_1(S)$,
    \begin{equation}\label{eq:singularvaluegap}
         \left| \log \left|\lambda_{i_1}(\rho_t(\gamma))\right|- \log \left|\lambda_{i_2}(\rho_t(\gamma))\right| -l \ell_X(\gamma) \right|\geqslant 2m_0t^2 \cdot \ell_X(\gamma).
    \end{equation}
    We discuss by three cases:
    \begin{enumerate}
        \item If $i_1 = i_2$, then the left-hand side of \prettyref{eq:singularvaluegap} is at least $\ell_X(\gamma)$, which is strictly greater than the right-hand side since $m_0 t^2<\frac{1}{8}$.
        \item If $\{i_1,i_2\} = \{1,2\}$ (the case for $\{i_1,i_2\} = \{2,3\}$ is identical by symmetry), then the left-hand side is at least $\ell_X(\gamma)/4$ by \prettyref{eq: rootpotential} and \prettyref{eq:smallsingularvaluevariation}, which is again strictly greater than the right-hand side since $2m_0 t^2<\frac{1}{4}$.
        \item If $\{i_1,i_2 \}=\{1,3\}$, we may assume without loss of generality that $i_1 = 1$ and $i_2 = 3$. We then consider two further subcases:
        \begin{itemize}
            \item If $l = 1$, then \prettyref{eq:singularvaluegap} follows directly from \prettyref{thm:Hilbert domination}.
            \item If $l \neq 1$, then \prettyref{eq:smallsingularvaluevariation}, together with the identity \[\log \left|\lambda_{1}(\rho_t(\gamma))\right|- \log \left|\lambda_{3}(\rho_t(\gamma))\right| = \int_{\ell_{\gamma}} \left(r_{1,t} + r_{2,t}\right) \dd \tau,\] imply that
            \[
            \frac{1}{2} \ell_X(\gamma)\leqslant\log \left|\lambda_{1}(\rho_t(\gamma))\right|- \log \left| \lambda_{3}(\rho_t(\gamma))\right| \leqslant \frac{3}{2}\ell_X(\gamma).            
            \]
            Since $l$ is an integer and $l \neq 1$, the left-hand side of \prettyref{eq:singularvaluegap} is at least $\frac{1}{2}\ell_X(\gamma)$, and is thus strictly greater than the right-hand side.
        \end{itemize}
    \end{enumerate}
    This concludes the proof of Borel Anosovness.

    Next, we show that the representation $\rho_t\otimes\operatorname{Sym}^{k-1}(j_X)$ is not Hitchin. Suppose it is for the sake of contradiction. Since Hitchin representations form a connected component, any continuous deformation of it would remain Hitchin, and hence Borel Anosov. However, $\rho_t\otimes\operatorname{Sym}^{k-1}(j_X)$ continuously deforms to $\rho_0\otimes\operatorname{Sym}^{k-1}(j_X)$, which is not Borel Anosov (see \prettyref{rem:non-borel-anosov}). Thus we get a contradiction and conclude the proof.

    Finally, note that for any $t\in \mathbb R$, $\rho_t$ does not belong to the Hitchin component, hence $\rho_t(\pi_1(S))$ is not contained in an irreducible $\mathrm{SL}_2 \mathbb R$. The irreduciblity of $\rho_t\otimes \operatorname{Sym}^{k-1} (j_X)$ thus comes from \prettyref{prop:irr}. 
\end{proof}

\subsection{Lyapunov exponents}

\begin{proof}[Proof of \prettyref{coro:lyapunov exponent}]
We recall the following two facts. 
\begin{proposition}\label{prop: lyapunovint}
\[
L_1(X,\rho_t) = \int_{T^1 S} f_t \dd m_X \text{ and } L_3(X,\rho_t) = -\int_{T^1 S} f_t\dd m_X.
\]
\end{proposition}
For a proof, we refer to Costantini--Martin-Baillon~\cite[Proposition~5.8, Lemma~5.9 and Lemma~5.10]{10.1093/imrn/rnae104} whose argument can be directly applied here. And this formula shows that we only need to carry out the proof for $L_1(X,\rho_t)$.

Next we recall the following equidistribution result:
\begin{theorem}[\cite{bowenperiodic}]\label{thm:equilibrium}
Let
\[
\mathcal P(T):=\left\{[\gamma]\ \middle|\gamma\in\pi_1(S)\setminus\{\mathrm{id}\}\text{ is primitive and } \ell_X(\gamma)\leq T\right\}
\]
where $[\gamma]$ denotes the conjugacy class of $\gamma$. Then
\[
m_X=\lim_{T\to\infty}\frac{1}{\#\mathcal P(T)}
\sum_{[\gamma]\in\mathcal P(T)}\widehat{\delta}_{[\gamma]}
\]
in the weak-$*$ topology, where $m_X$ is the normalized Liouville
measure and $\widehat{\delta}_{[\gamma]}$ is the flow invariant probability
measure supported on the oriented periodic orbit $\ell_\gamma$.
\end{theorem}

\vspace{3mm}

Keep using the notations in the proof of \prettyref{thm:Hilbert domination}, let $\varepsilon_2 = \varepsilon_0$ and $m_2 = m/8$ where $\varepsilon_0$ and $m$ are the constants therein. For each $N\in \mathbb{N}^{*}$ such that $\mathcal P(N)$ is non-empty, define the probability measure on $T^1 S$ by
\[
m_N:=\frac{1}{\#\mathcal P(N)}
\sum_{[\gamma]\in\mathcal P(N)}
\widehat{\delta}_{[\gamma]}.
\]
From \prettyref{prop: lyapunovint} and \prettyref{thm:equilibrium} we have
\[
L_1(X,\rho_t) = \lim_{N\to \infty} \int_{T^1 S} f_t \dd m_{N} =\lim_{N\to \infty} \frac{1}{\#\mathcal P(N)}
\sum_{[\gamma]\in\mathcal P(N)} \frac{\int_{\ell_{\gamma}} f_t\dd \tau}{\ell_X(\gamma)}.
\]
Applying \prettyref{eq: potentialvarint} and \prettyref{eq: remainderint} we get for every $t\in(-\varepsilon_2, \varepsilon_2)$ and $N\in \mathbb{N}^*$ sufficiently large,
\[
\frac{1}{\#\mathcal P(N)}\sum_{[\gamma]\in\mathcal P(N)} \frac{\int_{\ell_{\gamma}} f_t\dd \tau}{\ell_X(\gamma)}
\geqslant \frac{1}{\#\mathcal P(N)} \sum_{[\gamma]\in\mathcal P(N)}  \left(\frac{1}{2}+\frac{m}{8} t^2\right)  = \frac{1}{2}+m_2 t^2,
\]
which concludes the proof.
\end{proof}

\subsection{Hausdorff dimension of the limit sets}
For any $g\in \mathrm{SL}_3 \mathbb R$, let $\sigma_1(g)\geqslant \sigma_2(g)\geqslant \sigma_3(g)$ be the singular values of $g$. We first recall the definition of the affinity exponent (see Kaplan--Yorke~\cite{affinityexponentkaplanyorke}, Douady--Oesterl\'e~\cite{affinityexponentdouadyoesterle} and Falconer~\cite{falconerselfaffine} for more about it). Let $\rho:\pi_1(S)\to \mathrm{SL}_3\mathbb{R}$ be an Anosov representation. For $s\geqslant 1$, the \emph{affinity Poincaré series} of $\rho$ is defined by
\[
\Phi^{\mathrm{aff}}_{\rho}(s) := \sum_{\gamma\in \pi_1(S)} \frac{\sigma_2(\rho(\gamma))}{\sigma_1(\rho(\gamma))} \left(\frac{\sigma_3(\rho(\gamma))}{\sigma_1(\rho(\gamma))}\right)^{s-1},
\]
and the \emph{affinity exponent} $s_A(\rho)$ is defined as
\[
s_A(\rho) := \inf \left\{s\geqslant 1 \middle| \Phi^{\mathrm{aff}}_{\rho}(s)<\infty \right\}.
\]

Recall that $\Lambda^1_\rho\subset \mathbb{R}\mathbb{P}^2$ denotes the projective limit set of $\rho$. We have the following dimension equality:

\begin{theorem}[Li--Pan--Xu~{\cite[Theorem~1.2]{lipanxu}}]\label{thm:affgedim}
If $\rho$ is irreducible, then
\[
\dim_{\mathrm{H}}(\Lambda^1_{\rho}) = s_A(\rho).
\]
\end{theorem}
Note that the inequality $\dim_{\mathrm{H}}(\Lambda^1_{\rho}) \leqslant s_A(\rho)$ is proved for Anosov representations to $\mathrm{SL}_n \mathbb{R}$ by Pozzetti--Sambarino--Wienhard~\cite[Theorem~B]{liplim}.

\vspace{3mm}

Next we give a thermodynamic formalism description of the affinity exponent. Recall that for any H\"{o}lder continuous function $f$ on $T^1 S$, the \emph{topological pressure} of $f$ is defined by
\[
P(f): = \sup_{\mu \in \mathcal{M}}\left\{h_{\mu} +\int_{T^1 S} f \dd\mu\right\},
\]
where $\mathcal{M}$ is the space of flow invariant probability measures on $T^1 S$ and $h_\mu$ denotes the measure entropy of the geodesic flow with respect to $\mu$. Recall that if $D\subset\mathbb{R}^k$ is open and
$u\mapsto f_u$ is a real-analytic map from $D$ to $\mathcal{C}^\alpha(T^1S)$ for some fixed $\alpha>0$, then
\[
u\mapsto P(f_u)
\]
is real analytic, see \cite[Proposition~4.7]{parrypollicott} and
\cite[Corollary~5.27]{rullethermodynamic}.

Let $D\subset \mathbb{R}^k$ be an open neighbourhood of $\{0\}$, and let $\rho_u: \pi_1(S)\to \mathrm{SL}_3 \mathbb{R},\ u\in D$ be a family of Anosov representations varying analytically in $\mathrm{Hom}(\pi_1(S), \mathrm{SL}_3 \mathbb{R})$. Let $U$ and $\{f_u\}_{u\in U}\subset \mathcal{C}^{\alpha}(T^1 S)$ be the open set and family of functions in \prettyref{thm:bcls-variation}. For each $s\in \mathbb{R}$, define
\[
\mathfrak{f}_{s,u}: = (1+s) f_u+(s-2)f_u\circ \iota.  
\]
We can check for each nontrivial $\gamma\in \pi_1(S)$,
\begin{equation}\label{eq:integralperiod}    
\int_{\ell_{\gamma}} \mathfrak{f}_{s,u} \dd \tau =  -\log \left( \left| \frac{\lambda_2(\rho_u(\gamma))}{\lambda_1(\rho_u(\gamma))} \right|\cdot \left| \frac{\lambda_3(\rho_u(\gamma))}{\lambda_1(\rho_u(\gamma))} \right|^{s-1}\right).
\end{equation}

\begin{proposition}\label{prop: pressureandexponent}
For each $u\in U$, $s_A(\rho_u)$ is the unique solution of
\[
P(-\mathfrak{f}_{s,u}) = 0,
\]
where left-hand side is considered as a function of $s\geqslant 1$.
\end{proposition}
\begin{proof}

Since $\pi_1(S)$ is a closed surface group, the projective limit set of $\rho_u(\pi_1(S))$ is homeomorphic to the circle $S^1$ (see for example, Labourie~\cite{labourie2006anosov}), whose Hausdorff dimension is at least one. Thus, Pozzetti--Sambarino--Wienhard~\cite[Theorem~B]{liplim} shows that $s_A(\rho_u)\geqslant 1$. As a result, Li--Pan--Xu~\cite[Lemma~8.4]{lipanxu} shows that $s_{A}(\rho_u)$ is the unique number $s\geqslant 1$ such that
\[
h_{s,u}: = \limsup_{T\to \infty} \frac{1}{T} \log \# \left\{ [\gamma]\ \middle|\  \log \left( \left| \frac{\lambda_1(\rho_u(\gamma))}{\lambda_2(\rho_u(\gamma))} \right|\cdot \left| \frac{\lambda_1(\rho_u(\gamma))}{\lambda_3(\rho_u(\gamma))} \right|^{s-1}\right) \leqslant T \right\} = 1.
\]
On the other hand, \prettyref{eq:integralperiod} and Sambarino~\cite[Lemma~2.2.7, Corollary~5.5.3]{SAMBARINO_2024} show that for each $s\geqslant 1$, $h_{s,u}$ is the unique number such that
\[
P(-h_{s,u}\mathfrak{f}_{s,u}) =0.
\]
So we have concluded the proof.
    
\end{proof}

\begin{proof}[Proof of \prettyref{coro: affinity exponent}]
We continue to use the notations in the proof of \prettyref{thm:Hilbert domination} and \prettyref{coro:lyapunov exponent}. For $t\in(-\varepsilon_0,\varepsilon_0)$ and $s\geqslant1$, define
\[
\mathfrak{f}_{s,t}:=(1+s) f_t+(s-2)f_t\circ \iota.  
\]
Recall that $f_t$ is a real analytic family. So the function 
\[
P^{\mathrm{aff}}(s,t) := P(-\mathfrak{f}_{s,t})
\]
is also real analytic. Since $\rho_0 = j_X\oplus \mathbf{1}$, it is direct to check
\[
s_{A}(\rho_0) = \frac{3}{2},\quad P^{\mathrm{aff}}(s,0) = \frac{3}{2}-s.
\]
Regarding \prettyref{prop: pressureandexponent}, the implicit function theorem shows that for sufficiently small $t$, $s_{A}(\rho_t)$ depends analytically with respect to $t$.

Notice that
\[
\mathfrak{f}_{\frac{3}{2},t} = \frac{5}{2} f_t -\frac{1}{2} f_t\circ \iota.
\]
By Bridgeman--Canary--Labourie--Sambarino~\cite[Proposition~3.9]{BCLS}, we can compute the partial derivatives of $P^{\mathrm{aff}}$ with respect to $t$ in the following formula
\[
\partial_t P^{\mathrm{aff}}\left(\frac{3}{2},0\right) =-\int_{T^1S} \left( \frac{5}{2}\dot f-\frac{1}{2}(\dot f\circ\iota) \right) \dd m_X.
\]
(Note that to apply the cited formulas directly, $\rho_0 = j_X\oplus \mathbf{1}$ shows that $m_X$ is the equilibrium state of $-\mathfrak{f}_{\frac{3}{2},0}$ appearing there.)

\prettyref{eq: potentialvarint} and a similar application of \prettyref{thm:equilibrium} in the proof of \prettyref{coro:lyapunov exponent} imply that 
\[
\partial_t P^{\mathrm{aff}}\left(\frac{3}{2},0\right) = 0.
\]
As a result, we can further apply \cite[Proposition~3.9]{BCLS} with the chain rule to get
\[
\begin{aligned}
\partial_t^2 P^{\mathrm{aff}}\left(\frac32,0\right)
=&\lim_{T\to\infty}\frac1T\int_{T^1S}
\left(\int_0^T\left(\frac52 \dot f-\frac12 \dot f\circ\iota\right)
(\varphi^\tau(x))\dd\tau\right)^2\dd m_X(x) \\
&-\int_{T^1S}\left(\frac52 \ddot f-\frac12 \ddot f\circ\iota\right)\dd m_X.
\end{aligned}
\]
The first term vanishes due to \prettyref{eq: potentialvarint} and Livšic~\cite{Livšic_1972} (for convenience, see \cite[Section~3.1.2 and Theorem~3.3]{BCLS}, which shows that there exists a H\"{o}lder function $\mathcal{U}$ on $T^1 S$  so that it is $C^1$ in the flow direction and
\[
\int_0^T\left(\frac52 \dot f-\frac12 \dot f\circ\iota\right)
(\varphi^\tau(x))\dd\tau =\mathcal{U}(\varphi^T(x)) - \mathcal{U}(x)
\]
holds for every $x\in T^1 S$.)

For the second term, \prettyref{prop: lyapunovint} shows that 
\[
\partial_t^2 P^{\mathrm{aff}}\left(\frac32,0\right) = -2\int_{T^1 S} \ddot f \dd m_X =  -2\frac{\dd^2}{\dd t^2} L_1(X,\rho_t) \big|_{t = 0}.
\]
\prettyref{coro:lyapunov exponent} shows that $\partial_t^2 P^{\mathrm{aff}}\left(3/2,0\right)< - m'$ for some $m'>0$ depending only on $q$. As a result, again by the implicit function theorem, there exists some $m''>0$ depending only on $q$ such that
\begin{equation}\label{eq:derivativesoftheaffinityexponent}
s_{A}(\rho_0) = \frac{3}{2},\quad \frac{\dd}{\dd t} s_A(\rho_t)|_{t=0} = 0,\quad \frac{\dd^2}{\dd t^2} s_A(\rho_t)|_{t=0} <-m''.    
\end{equation}
\prettyref{coro:irr} shows that whenever $t\not = 0$, $\rho_t$ is irreducible, and thus by \prettyref{thm:affgedim} we have
\[
\dim_{\mathrm{H}}(\Lambda^1_{\rho_t}) = s_{A}(\rho_t)
\]
whenever $t\not = 0$. The existence of $\varepsilon_3, m_3$ and $m'_3$ thus comes from \prettyref{eq:derivativesoftheaffinityexponent}.
\end{proof}

\begin{remark}
We continue to use the notations in \prettyref{prop: pressureandexponent}. Let $D\subset \mathbb{R}^k$ be an open neighbourhood of $\{0\}$, and let $\rho_u: \pi_1(S)\to \mathrm{SL}_3 \mathbb{R}$, $\ u\in D$ be a family of Anosov representations varying analytically in $\mathrm{Hom}(\pi_1(S), \mathrm{SL}_3 \mathbb{R})$. Let $U$ and $\{f_u\}_{u\in U}\subset \mathcal{C}^{\alpha}(T^1 S)$ be the open set and family of functions in \prettyref{thm:bcls-variation}.

As above, \cite[Proposition~3.9]{BCLS} implies that
\[
\partial_s P( -\mathfrak{f}_{s,u})|_{(s,u) = (s_A(\rho_0), 0)} = -\int_{T^1 S} \left(f_0+f_0\circ \iota \right)  \dd m_{eq},
\]
where $m_{eq}$ is the equilibrium state measure for $-\mathfrak{f}_{s_A(\rho_0),0}$. Since $f_0$ is positive, we have \[\partial_s P( -\mathfrak{f}_{s,u})|_{(s,u) = (s_A(\rho_0), 0)}\neq 0,\] \prettyref{prop: pressureandexponent} and the implicit function theorem show that $s_{A}(\rho_u)$ depends analytically on $u$, which addresses \cite[Question~1.11]{lipanxu} when the Anosov representation is from a closed surface group.
\end{remark}

\begin{remark}\label{rem: dimension>3/2}
As mentioned in the Introduction, we give the example of small irreducible defomrations $\rho:\pi_1(S)\to \mathrm{SL}_{3}\mathbb{R}$ of $\rho_0 = j_X\oplus \mathbf{1}$ satisfying
\[
\dim_{\mathrm{H}} (\Lambda^1_\rho)>\frac{3}{2}.
\]

Firstly, pick a smooth closed $1$-form $\omega$ on $S$ whose cohomology class is nontrivial, which determines a smooth function $p_{\omega}$ on $T^1 S$ by $p_{\omega}(x,v) = \omega_x(v)$. Notice that $p_{\omega}\circ \iota = -p_{\omega}$. The map
\[
\chi(\gamma):= \int_{\ell_{\gamma}} p_\omega \dd \tau, \quad \gamma\in \pi_1(S)
\]
is thus a nontrivial group homomorphism. For each $\nu\in \mathbb{R}$, define the family of representations
\[
\rho_{\nu}(\gamma):=e^{\nu\cdot \chi(\gamma)} j_X(\gamma)\oplus e^{-2\nu\cdot \chi(\gamma)},\quad \gamma\in \pi_1(S).
\]
Notice that $\left|\chi(\gamma)\right|\leqslant \|p_{\omega}\|_{L^{\infty}} \ell_X(\gamma)$, so by \prettyref{defn:eigenvalue-gap} there exists $\tilde{\varepsilon}>0$ such that for any $\nu \in (-\tilde{\varepsilon}, \tilde{\varepsilon})$, $\rho_\nu$ is Anosov and satisfies for any $\gamma\in \pi_1(S)$,
\[
\log |\lambda_1(\rho_{\nu}(\gamma))| = \frac{1}{2}\ell_X(\gamma)+\nu\chi(\gamma)\geqslant  0.
\]
Define $f_{\nu} = \frac{1}{2}+\nu p_{\omega}$ which is an analytic family satisfying for any nontrivial $\gamma\in \pi_1(S)$,
\[
\int_{\ell_{\gamma}} f_{\nu} \dd \tau = \log |\lambda_1(\rho_{\nu}(\gamma))|.
\]
And for $s\in\mathbb{R}$, define
\[
\mathfrak{f}_{s,\nu}: = (1+s) f_\nu+(s-2)f_\nu\circ \iota=s-\frac{1}{2}+3\nu p_{\omega}.
\]

From \prettyref{prop: pressureandexponent} we see for each $\nu\in (-\tilde{\varepsilon},\tilde{\varepsilon})$, $s_{A}(\rho_\nu)$ is the unique zero for the function
\[
P^{\mathrm{aff}}(s,\nu) := P(-\mathfrak{f}_{s,\nu}),\quad s\geqslant 1.
\]
Following similar arguments from the proof of \prettyref{coro: affinity exponent}, we see
\[
s_{A}(\rho_0) =\frac{3}{2},\quad \partial_{s} P^{\mathrm{aff}}(s,\nu)|_{(s,\nu) = (\frac{3}{2},0)} = -1,
\]
\[
\quad \partial_{\nu} P^{\mathrm{aff}}(s,\nu)|_{(s,\nu) = (\frac{3}{2},0)} = -\int_{T^1 S} 3 p_{\omega} \dd m_X =0 \text{ (since } p_{\omega}\circ \iota = -p_{\omega})  ,
\]
and
\[
\partial^2_{\nu} P^{\mathrm{aff}}(s,\nu)|_{(s,\nu) = (\frac{3}{2},0)} = \lim_{T\to\infty}\frac1T\int_{T^1S}
\left(\int_0^T3p_{\omega}\left(\varphi^\tau(x)\right)\dd\tau\right)^2\dd m_X(x).
\]
The last term is nonzero due to  \cite[Proposition~3.9~(4)]{BCLS} and $\omega$ has nontrivial cohomology class. By the implicit function theorem, $s_A(\rho_\nu)$ varies analytically with respect to $\nu$ near $0$, whose first derivative is $0$ and second derivative is positive. As a result, we can find $\nu_0$ arbitrarily small such that $s_{A}(\rho_{\nu_0})>\frac{3}{2}$, and sufficiently small irreducible deformations of $\rho_{\nu_0}$ provide the example.

\end{remark}


\appendix

\section{Irreducibility of the tensor product}

In this appendix, we prove the following proposition.

\begin{proposition}\label{prop:irr}
    Let $\rho: \pi_1(S)\to \mathrm{SL}_3 \mathbb R$ be an irreducible Borel Anosov representation such that $\rho(\pi_1(S))$ in not contained in an irreducible $\mathrm{SL}_2 \mathbb R$ and let $j_X\colon\pi_1(S)\to\mathrm{SL}_2\mathbb{R}$ denote the Fuchsian representation given by $\KK_X^{1/2}$, then 
    \[R_{\rho,k}:=\rho\otimes\operatorname{Sym}^{k-1}(j_X)\colon\pi_1(S)\to\mathrm{SL}_{3k}\mathbb{R}\] is irreducible for every $k\geqslant1$.
\end{proposition}

\begin{proof}

    We first show that $\rho(\pi_1(S))$ is Zariski dense in $\mathrm{SL}_3\mathbb R$. Let $G_\rho$ be the Zariski closure of $\rho(\pi_1(S))$ in $\mathrm{SL}_3\mathbb R$. In particular, $G_\rho$ is closed
    in the Euclidean topology.  By \cite[Corollary 2.20]{BCLS} $G_\rho$ is semisimple without compact factors with center $Z(G_\rho)=\{ I\}$. Set $\mathfrak g_\rho:=\operatorname{Lie}(G_\rho)\subset\mathfrak{sl}_3\mathbb R$. Since $\mathfrak g_\rho$ is semisimple, its representation is completely reducible.
    
    $\mathfrak g_\rho$ acts irreducibly on $\mathbb R^3$. Otherwise, it preserves either a direct sum of a two-dimensional irreducible summand and a one-dimensional trivial summand, or a direct sum of three one-dimensional summands. Since 1-dimensional representations of semisimple Lie algebras are trivial, we only need to discuss the former case. In this case, the 2-dimensional summand which equals $\mathfrak g_\rho\cdot \mathbb R^3$, is $G_\rho$-invariant. This contradicts the irreducibility of $\rho(\pi_1(S))$.
    
   The complexification of this representation is irreducible.
    Indeed, if $W\subsetneq\mathbb C^3$ were a nonzero
    $\mathfrak g_{\rho,\mathbb C}$-invariant subspace,
    real irreducibility would imply $W\cap\overline W=0$ and $W+\overline W=\mathbb C^3$, contradicting $3=2\dim_{\mathbb C}W$. By the classification of complex semisimple Lie algebras and their representations, either
    $\mathfrak g_{\rho,\mathbb C}=\mathfrak{sl}_3\mathbb C$ or
    $\mathfrak g_{\rho,\mathbb C}$ is conjugate to the principal
    $\mathfrak{sl}_2\mathbb C$, acting through
    $\operatorname{Sym}^2(\mathbb C^2)$.
    
    We claim that $\mathfrak g_{\rho,\mathbb C}\cong \mathfrak{sl}_3\mathbb C$, which means  $\mathfrak{g}_\rho=\mathfrak{sl}_3\mathbb R$ since $\mathfrak{g}_\rho\subset\mathfrak{sl}_3\mathbb R$, and thus $G_\rho=\mathrm{SL}_3\mathbb R$ since the latter is connected. Otherwise, $\mathfrak g_\rho$ is a real form of $\mathfrak{sl}_2\mathbb C$ acting irreducibly on $\mathbb R^3$, which is conjugate to either the principal $\mathfrak{sl}_2\mathbb R$ or $\mathfrak{su}_2 \cong \mathfrak{so}_3 \mathbb R$. In the former case is excluded by the assumption. In the latter case, $\mathfrak{su}_2 $ is compact and excluded by \cite[Corollary 2.20]{BCLS}.
    
    Now consider the representation
        \[\widehat\rho:=(\rho, j_X)\colon
        \pi_1(S)\longrightarrow
        \mathrm{SL}_3\mathbb R\times\mathrm{SL}_2\mathbb R,\]
    and let $H$ be its Zariski closure in $\mathrm{SL}_3\mathbb R\times\mathrm{SL}_2\mathbb R$ which is also Euclidean closed. Let
    \[
    p_1\colon H\longrightarrow \mathrm{SL}_3\mathbb R,
    \qquad
    p_2\colon H\longrightarrow \mathrm{SL}_2\mathbb R
    \]
    be the two projections. Since $p_1(H)\supset p_1(\widehat\rho(\pi_1(S)))=\rho(\pi_1(S))$, it is Zariski dense in $\mathrm{SL}_3\mathbb{R}$. Therefore $p_1$ is dominant as a morphism between algebraic groups, hence \[\dd p_1\colon\mathfrak{h}:=\operatorname{Lie}(H)\longrightarrow\mathfrak{sl}_3\mathbb{R}\] is surjective (see Milne~\cite[Corollaries~1.69 and~3.25]{milne2017}). Similarly, since $j_X(\pi_1(S))$ is Zariski dense in $\mathrm{SL}_2\mathbb{R}$, $\dd p_2\colon\mathfrak{h}\to\mathfrak{sl}_2\mathbb{R}$ is also surjective.

    Set
    \[
    \mathfrak m_1:=\dd p_1\left(\mathfrak h\cap(\mathfrak{sl}_3\mathbb R \oplus 0)\right), \quad \mathfrak m_2:=\dd p_2\left(\mathfrak h\cap(0\oplus\mathfrak{sl}_2\mathbb R)\right).
    \]
    We claim that they are ideals of $\mathfrak{sl}_3\mathbb R$ and $\mathfrak{sl}_2\mathbb R$, respectively. We give the argument only for $\mathfrak{m}_1$; the argument for \(\mathfrak m_2\) is analogous.  Let $Y\in\mathfrak m_1$ and $Z\in\mathfrak{sl}_3\mathbb R$. By the surjectivity of $\dd p_1$, there exists $U\in\mathfrak{sl}_2\mathbb R$ such that $(Z,U)\in\mathfrak h$. Since $(Y,0)\in\mathfrak h$, we have
    \[
    [(Z,U),(Y,0)]=([Z,Y],0)\in\mathfrak h.
    \]
    Therefore $[Z,Y]\in\mathfrak m_1$. So $\mathfrak m_1, \mathfrak m_2$ are both ideals.
    
    We show that $\mathfrak m_1= \mathfrak{sl}_3\mathbb R$ and $\mathfrak m_2 = \mathfrak{sl}_2\mathbb R$. Since $\mathfrak{sl}_2\mathbb R$ and $\mathfrak{sl}_3\mathbb R$ are simple, we only need to rule out the possibility that one of them is a zero ideal. If $\mathfrak m_1=0$, then $\dd p_2\colon\mathfrak h\to\mathfrak{sl}_2\mathbb R$ is both injective and surjective, hence an isomorphism. Therefore $\dd p_1\circ \left(\dd p_2 \right)^{-1} \colon \mathfrak{sl}_2\mathbb R \to\mathfrak{sl}_3\mathbb R$ is a surjective Lie algebra homomorphism, which is impossible. Similarly, if $\mathfrak m_2 = 0$, then $\dd p_2\circ \left(\dd p_1\right)^{-1}\colon \mathfrak{sl}_3\mathbb R \to\mathfrak{sl}_2\mathbb R$ is a surjective Lie algebra homomorphism, which is also impossible.

     Hence, $\mathfrak h=\mathfrak{sl}_3\mathbb R\oplus\mathfrak{sl}_2\mathbb R$. Since $\mathrm{SL}_3\mathbb R\times\mathrm{SL}_2\mathbb R$ is connected, $H=\mathrm{SL}_3\mathbb R\times\mathrm{SL}_2\mathbb R$.

    We now prove the irreducibility of $R_{\rho,k}$. To be brief, set $W_k=\operatorname{Sym}^{k-1}(\mathbb R^2)$. Consider the representation
    \[
    \Pi_k:\mathrm{SL}_3\mathbb R\times\mathrm{SL}_2\mathbb R
    \longrightarrow \mathrm{GL}(\mathbb R^3\otimes_{\mathbb R} W_k)
    \]
    defined by
    \[
    \Pi_k(A,B)=A\otimes_{\mathbb R}\operatorname{Sym}^{k-1}(B).
    \]
    Since its complexification 
    \[\Pi_{k,\mathbb{C}}\colon \ \mathrm{SL}_3\mathbb C\times\mathrm{SL}_2\mathbb C
    \longrightarrow
    \mathrm{GL}\left((\mathbb R^3\otimes_{\mathbb R}\mathbb{C})\otimes_{\mathbb{C}} (W_k\otimes_{\mathbb R} \mathbb{C})\right)\]
    is the external tensor product of two irreducible representations, $\Pi_{k,\mathbb{C}}$ is also irreducible (see Goodman--Wallach~\cite[Section~4.2.2]{goodman}). It follows that $\Pi_k$ is also
    irreducible; otherwise a nonzero proper real invariant subspace would complexify to a nonzero proper complex invariant subspace.

    Notice that $R_{\rho,k} = \Pi_k \circ \widehat\rho$. Suppose that there exists a nonzero subspace $U\subsetneq \mathbb{R}^3\otimes_{\mathbb R} W_k$ invariant under $R_{\rho,k}(\pi_1(S))$. Define
    \[
    \operatorname{Stab}(U) =\left\{(A,B)\in\mathrm{SL}_3\mathbb R\times\mathrm{SL}_2\mathbb R \ \middle | \ \Pi_k(A,B)(U)=U \right\}
    \]
    which is a Zariski closed algebraic subgroup containing $\widehat \rho(\pi_1(S))$. Thus $\operatorname{Stab}(U) = \mathrm{SL}_3\mathbb R\times \mathrm{SL}_2 \mathbb R$, contradicting the irreduciblity of $\Pi_k$.
\end{proof}


\bibliographystyle{alpha}
\bibliography{bib}

\end{document}